\documentclass[a4paper,11pt,reqno]{amsart}

\usepackage{amsmath,amssymb,amsthm}
\usepackage{tikz}
\usepackage{graphicx}
\usepackage{hyperref}
\usepackage[margin=3cm]{geometry}

\newtheorem{thm}{Theorem}[section]
\newtheorem{lem}[thm]{Lemma}
\newtheorem{prop}[thm]{Proposition}
\newtheorem{rem}[thm]{Remark}
\newtheorem{cor}[thm]{Corollary}
\newtheorem{conj}[thm]{Conjecture}

\theoremstyle{definition}
\newtheorem{defi}{Definition}[section]

\newcommand{\ii}{\mathbf{i}}
\newcommand{\N}{\mathbb{N}}
\newcommand{\n}{\mathbf{n}}
\newcommand{\x}{\mathbf{x}}
\DeclareMathOperator{\dotcup}{\dot\bigcup}
\DeclareMathOperator{\FCP}{FCP}
\DeclareMathOperator{\SPP}{SPP}
\DeclareMathOperator{\PP}{PP}
\DeclareMathOperator{\TCPP}{TCPP}
\DeclareMathOperator{\QS}{QS}
\DeclareMathOperator{\qspp}{qspp}
\DeclareMathOperator{\qtcpp}{qtcpp}
\DeclareMathOperator{\qtcspp}{qtcspp}

\newmuskip\pFqmuskip

\newcommand*\pFq[6][8]{%
	\begingroup % only local assignments
	\pFqmuskip=#1mu\relax
	\mathchardef\normalcomma=\mathcode`,
	\mathcode`\,=\string"8000
	\begingroup\lccode`\~=`\,
	\lowercase{\endgroup\let~}\pFqcomma
	{}_{#2}F_{#3}{\left[\left.\genfrac..{0pt}{}{#4}{#5}\right|#6\right]}%
	\endgroup
}
\newcommand{\pFqcomma}{{\normalcomma}\mskip\pFqmuskip}

\definecolor{darkblue}{rgb}{0.0,0,0.7} % darkblue color
\newcommand{\darkblue}{\color{darkblue}} % darkblue command
\newcommand{\defn}[1]{\emph{\darkblue #1}} % emphasis of a definition

\newcommand{\RightBoundaryColour}[4]{
\begin{scope}[xshift= 1cm*cos(30)*(#1+#2) , yshift=1cm*#3+1cm*sin(30)*(#2-#1)]
\draw [fill=#4] (0,0) -- ({1*cos(30)},{1*sin(30)}) -- ({1*cos(30)},{-1*sin(30)}) -- (0,-1) -- (0,0);
\end{scope}
}
\newcommand{\LeftBoundaryColour}[4]{
\begin{scope}[xshift= 1cm*cos(30)*(#1+#2) , yshift=1cm*#3+1cm*sin(30)*(#2-#1)]
\draw (0,0)[fill=#4] -- ({-1*cos(30)},{1*sin(30)}) -- ({-1*cos(30)},{-1*sin(30)}) -- (0,-1) -- (0,0);
\end{scope}
}
\newcommand{\TopBoundaryColour}[4]{
\begin{scope}[xshift= 1cm*cos(30)*(#1+#2) , yshift=1cm*#3+1cm*sin(30)*(#2-#1)]
\draw (0,0)[fill=#4] -- ({1*cos(30)},{1*sin(30)}) -- (0,1) -- ({-1*cos(30)},{1*sin(30)}) -- (0,0);
\end{scope}
}

\newcommand{\RightBoundaryTriangle}[3]{
\begin{scope}[xshift= 1cm*cos(30)*(#1+#2) , yshift=1cm*#3+1cm*sin(30)*(#2-#1)]
\draw [fill=white] (0,0) -- ({1*cos(30)},{1*sin(30)}) -- ({1*cos(30)},{-1*sin(30)}) -- (0,-1) -- (0,0);
\draw [densely dotted] (0,0) -- ({1*cos(30)},{-1*sin(30)});
\end{scope}
}
\newcommand{\LeftBoundaryTriangle}[3]{
\begin{scope}[xshift= 1cm*cos(30)*(#1+#2) , yshift=1cm*#3+1cm*sin(30)*(#2-#1)]
\draw (0,0)[fill=white] -- ({-1*cos(30)},{1*sin(30)}) -- ({-1*cos(30)},{-1*sin(30)}) -- (0,-1) -- (0,0);
\draw[densely dotted] (0,0) -- ({-1*cos(30)},{-1*sin(30)});
\end{scope}
}
\newcommand{\TopBoundaryTriangle}[3]{
\begin{scope}[xshift= 1cm*cos(30)*(#1+#2) , yshift=1cm*#3+1cm*sin(30)*(#2-#1)]
\draw (0,0)[fill=white] -- ({1*cos(30)},{1*sin(30)}) -- (0,1) -- ({-1*cos(30)},{1*sin(30)}) -- (0,0);
\draw[densely dotted] (0,0) -- (0,1);
\end{scope}
}

\newcommand{\RightBoundaryTriangleDotted}[3]{
\begin{scope}[xshift= 1cm*cos(30)*(#1+#2) , yshift=1cm*#3+1cm*sin(30)*(#2-#1)]
\draw [fill=white, densely dotted] (0,0) -- ({1*cos(30)},{1*sin(30)}) -- ({1*cos(30)},{-1*sin(30)}) -- (0,-1) -- (0,0);
\draw [densely dotted] (0,0) -- ({1*cos(30)},{-1*sin(30)});
\end{scope}
}
\newcommand{\LeftBoundaryTriangleDotted}[3]{
\begin{scope}[xshift= 1cm*cos(30)*(#1+#2) , yshift=1cm*#3+1cm*sin(30)*(#2-#1)]
\draw (0,0)[fill=white, densely dotted] -- ({-1*cos(30)},{1*sin(30)}) -- ({-1*cos(30)},{-1*sin(30)}) -- (0,-1) -- (0,0);
\draw[densely dotted] (0,0) -- ({-1*cos(30)},{-1*sin(30)});
\end{scope}
}
\newcommand{\TopBoundaryTriangleDotted}[3]{
\begin{scope}[xshift= 1cm*cos(30)*(#1+#2) , yshift=1cm*#3+1cm*sin(30)*(#2-#1)]
\draw (0,0)[fill=white, densely dotted] -- ({1*cos(30)},{1*sin(30)}) -- (0,1) -- ({-1*cos(30)},{1*sin(30)}) -- (0,0);
\draw[densely dotted] (0,0) -- (0,1);
\end{scope}
}

\newcommand{\RightBoundaryMatching}[3]{
\begin{scope}[xshift= 1cm*cos(30)*(#1+#2) , yshift=1cm*#3+1cm*sin(30)*(#2-#1)]
\draw  ({1/3*cos(30)},{-1/3*(1+1*sin(30))}) -- ({2/3*cos(30)},0);
\filldraw ({1/3*cos(30)},{-1/3*(1+1*sin(30))}) circle (2pt);
\filldraw ({2/3*cos(30)},0) circle (2pt);
\end{scope}
}
\newcommand{\LeftBoundaryMatching}[3]{
\begin{scope}[xshift= 1cm*cos(30)*(#1+#2) , yshift=1cm*#3+1cm*sin(30)*(#2-#1)]
\draw ({-2/3*cos(30)},0) -- ({-1/3*cos(30)},{-1/3*(1+1*sin(30))});
\filldraw ({-2/3*cos(30)},0) circle (2pt);
\filldraw ({-1/3*cos(30)},{-1/3*(1+1*sin(30))}) circle (2pt);
\end{scope}
}
\newcommand{\TopBoundaryMatching}[3]{
\begin{scope}[xshift= 1cm*cos(30)*(#1+#2) , yshift=1cm*#3+1cm*sin(30)*(#2-#1)]
\draw ({1/3*cos(30)},{1/3*(1+1*sin(30))}) -- ({-1/3*cos(30)},{1/3*(1+1*sin(30))});
\filldraw ({1/3*cos(30)},{1/3*(1+1*sin(30))}) circle (2pt);
\filldraw ({-1/3*cos(30)},{1/3*(1+1*sin(30))}) circle (2pt);
\end{scope}
}

\newcommand{\RightBoundary}[4]{
\ifnum #4 =1
	\RightBoundaryColour{#1}{#2}{#3}{green!80!black}
\else
	\ifnum #4 =2
		\RightBoundaryColour{#1}{#2}{#3}{darkgray}
	\else
		\ifnum #4=3
			\RightBoundaryColour{#1}{#2}{#3}{white}	
		\else
			\ifnum #4=4
				\RightBoundaryMatching{#1}{#2}{#3}	
			\else
				\ifnum #4=5
					\RightBoundaryTriangle{#1}{#2}{#3}
				\else
					\ifnum #4=6
						\RightBoundaryTriangleDotted{#1}{#2}{#3}
					\else
						\ifnum #4=7
							\RightBoundaryColour{#1}{#2}{#3}{white}
							\begin{scope}[xshift= 1cm*cos(30)*(#1+#2) , yshift=1cm*#3+1cm*sin(30)*(#2-#1)]
								\draw[cyan, line width=2pt]  ({1/2*cos(30)},{1/2*sin(30)}) -- ({1/2*cos(30)},{-1+1/2*sin(30)});
							\end{scope}
						\else
							\ifnum #4=8
								\RightBoundaryColour{#1}{#2}{#3}{white}
							\else
								\ifnum #4=9
									\RightBoundaryColour{#1}{#2}{#3}{white}
									\begin{scope}[xshift= 1cm*cos(30)*(#1+#2) , yshift=1cm*#3+1cm*sin(30)*(#2-#1)]
										\draw[red, dashed, line width=2pt]  ({1/2*cos(30)},{1/2*sin(30)}) -- ({1/2*cos(30)},{-1+1/2*sin(30)});
									\end{scope}
								\else
									\RightBoundaryColour{#1}{#2}{#3}{white}
								\fi
							\fi
						\fi
					\fi
				\fi
			\fi
		\fi
	\fi
\fi
}
\newcommand{\LeftBoundary}[4]{
\ifnum #4 =1
	\LeftBoundaryColour{#1}{#2}{#3}{blue}
\else
	\ifnum #4 =2
		\LeftBoundaryColour{#1}{#2}{#3}{lightgray}
	\else
		\ifnum #4=3
			\LeftBoundaryColour{#1}{#2}{#3}{white}	
		\else
			\ifnum #4=4
				\LeftBoundaryMatching{#1}{#2}{#3}	
			\else
				\ifnum #4=5
					\LeftBoundaryTriangle{#1}{#2}{#3}
				\else
					\ifnum #4=6
						\LeftBoundaryTriangleDotted{#1}{#2}{#3}
					\else 
						\ifnum #4=7
							\LeftBoundaryColour{#1}{#2}{#3}{white}	
						\else
							\ifnum #4=8
								\LeftBoundaryColour{#1}{#2}{#3}{white}	
								\begin{scope}[xshift= 1cm*cos(30)*(#1+#2) , yshift=1cm*#3+1cm*sin(30)*(#2-#1)]
									\draw[cyan, line width=2pt]  ({-1/2*cos(30)},{1/2*sin(30)}) -- ({-1/2*cos(30)},{-1+1/2*sin(30)});
								\end{scope}
							\else
								\ifnum #4=9
									\LeftBoundaryColour{#1}{#2}{#3}{white}
								\else
									\LeftBoundaryColour{#1}{#2}{#3}{white}	
									\begin{scope}[xshift= 1cm*cos(30)*(#1+#2) , yshift=1cm*#3+1cm*sin(30)*(#2-#1)]
										\draw[red, dashed, line width=2pt]  ({-1/2*cos(30)},{1/2*sin(30)}) -- ({-1/2*cos(30)},{-1+1/2*sin(30)});
									\end{scope}
								\fi
							\fi
						\fi
					\fi
				\fi
			\fi
		\fi
	\fi
\fi
}
\newcommand{\TopBoundary}[4]{
\ifnum #4 =1
	\TopBoundaryColour{#1}{#2}{#3}{red}
\else
	\ifnum #4 =2
		\TopBoundaryColour{#1}{#2}{#3}{white}
	\else
		\ifnum #4=3
			\TopBoundaryColour{#1}{#2}{#3}{white}	
		\else
			\ifnum #4=4
				\TopBoundaryMatching{#1}{#2}{#3}	
			\else
				\ifnum #4=5
					\TopBoundaryTriangle{#1}{#2}{#3}
				\else
					\ifnum #4=6
						\TopBoundaryTriangleDotted{#1}{#2}{#3}
					\else
						\ifnum #4=7
							\TopBoundaryColour{#1}{#2}{#3}{white}
							\begin{scope}[xshift= 1cm*cos(30)*(#1+#2) , yshift=1cm*#3+1cm*sin(30)*(#2-#1)]
								\draw[cyan, line width=2pt]  ({-1/2*cos(30)},{1-1/2*sin(30)}) -- ({1/2*cos(30)},{1/2*sin(30)});
							\end{scope}
						\else
							\ifnum #4=8
								\TopBoundaryColour{#1}{#2}{#3}{white}
								\begin{scope}[xshift= 1cm*cos(30)*(#1+#2) , yshift=1cm*#3+1cm*sin(30)*(#2-#1)]
									\draw[cyan, line width=2pt]  ({-1/2*cos(30)},{1/2*sin(30)}) -- ({1/2*cos(30)},{1-1/2*sin(30)});
								\end{scope}
							\else
								\ifnum #4=9
									\TopBoundaryColour{#1}{#2}{#3}{white}
									\begin{scope}[xshift= 1cm*cos(30)*(#1+#2) , yshift=1cm*#3+1cm*sin(30)*(#2-#1)]
										\draw[red, dashed, line width=2pt]  ({-1/2*cos(30)},{1-1/2*sin(30)}) -- ({1/2*cos(30)},{1/2*sin(30)});
									\end{scope}
								\else
									\TopBoundaryColour{#1}{#2}{#3}{white}
									\begin{scope}[xshift= 1cm*cos(30)*(#1+#2) , yshift=1cm*#3+1cm*sin(30)*(#2-#1)]
										\draw[red, dashed, line width=2pt]  ({-1/2*cos(30)},{1/2*sin(30)}) -- ({1/2*cos(30)},{1-1/2*sin(30)});
									\end{scope}
								\fi
							\fi
						\fi
					\fi
				\fi
			\fi
		\fi
	\fi
\fi
}

	\newcounter{x}
	\newcounter{y}
	\newcounter{z}
	\newcounter{help}

\newcommand{\PlanePartitionColour}[2]{
	\foreach \m [count=\y] in {#1}{		 								%\m = row, \y = number of the row
		\foreach \z [count=\x] in \m{									%\z = entry in the row, \x the position of the entry
			\setcounter{x}{\x}											%remembers the position of the current entry
			\ifnum \z > 0												% if the entry is positive
				\TopBoundary{\x}{-\y}{\z}{#2}								%draw its top
			\fi
			
			\ifnum \x > 1												%if we are not regarding the first entry of the row
				\ifnum \z < \thez										% and if the current entry is less than the previous			
					\setcounter{help}{\z}
					\addtocounter{help}{1}								%has to be done this way since an update (30.10.2019)
					\foreach \zz in {\thehelp,...,\thez}{					
						\RightBoundary{(\x -1)}{(-\y)}{\zz}{#2}				%draw the correct amount of right sides of the boxes in
					}														% the previous position.
				\fi
			\fi
			\ifnum \y > 1											%if we are at least in the second row
				\foreach \mm [count=\yy] in {#1}{					%\mm = another row and \yy is its number
				\ifnum \yy = \they									%if \mm is the row before \m							
				\foreach \zz [count=\xx] in \mm{					%let \zz be the entry in \mm and \xx its position
				\ifnum \xx = \x										%if the positions \xx and \x coincide
				\ifnum \zz > \z										% and the entry \zz is larger than \z
					\setcounter{help}{\zz}									%has to be done like this since the update
					\addtocounter{help}{-1}	
					\foreach \zzz in {\z,...,\thehelp}{
						\LeftBoundary{(\x -1)}{(1-\yy)}{\zzz}{#2}		% draw the correct amount of left sides for the boxes "\zz"
					}
				\fi
				\fi
				}
				\fi
				}
			\fi		
			\setcounter{z}{\z}										%remembers the previous entry	
		}										%done with all entries in row \m
		\ifnum \thez >0												%if the last entry is larger than 0
			\foreach \z in {1,...,\thez}{
				\RightBoundary{\thex}{-\y}{\z}{#2}						%draw the appropriate number of left sides of these boxes
			}
		\fi
		\setcounter{y}{\y}
	}											%done with row \m

	\foreach \mm [count=\yy] in {#1}{								%regards the left sides of the boxes in the last row
	\ifnum \yy = \they
	\foreach \zz [count=\xx] in \mm{
	\ifnum \zz >0
		\foreach \zzz in {1,...,\zz}{
			\LeftBoundary{(\xx -1)}{(1-\yy)}{(\zzz-1)}{#2}			%draw the appropriate number of left sides of the boxes
		}
	\fi
	}
	\fi
	}
	
}

\newcommand{\TilingBox}[4]{
	\foreach \a in {1,...,#2}{
		\foreach \b in {1,...,#3}{
			\RightBoundary{0}{-\a}{\b -1}{#4}
		}
	}
	\foreach \a in {1,...,#1}{
		\foreach \b in {1,...,#3}{
			\LeftBoundary{\a}{0}{\b}{#4}
		}
	}
	\foreach \a in {1,...,#1}{
		\foreach \b in {1,...,#2}{
			\TopBoundary{\a}{-\b}{0}{#4}
		}
	}
}

\newcommand{\PlanePartition}[1]{\PlanePartitionColour{#1}{1}}
\newcommand{\PlanePartitionWhite}[1]{\PlanePartitionColour{#1}{3}}

\numberwithin{equation}{section}

\title{Fully complementary higher dimensional partitions}

\author{Florian Schreier-Aigner}
\address{University of Vienna, Austria}
\urladdr{\url{https://homepage.univie.ac.at/florian.schreier-aigner}}
\thanks{Florian Schreier-Aigner acknowledges the financial support from the Austrian Science Foundation FWF, grant J 4387}

\begin{document}

\begin{abstract}
%We present a new symmetry class for plane partitions - quarter complementary plane partitions - which can be extended to
We introduce a symmetry class for higher dimensional partitions - \emph{fully complementary higher dimensional partitions} (FCPs) - and prove a formula for their generating function. By studying symmetry classes of FCPs in dimension 2, we define variations of the classical symmetry classes for plane partitions. As a by-product we obtain conjectures for three new symmetry classes of plane partitions and prove that another new symmetry class, namely \emph{quasi transpose complementary plane partitions} are equinumerous to symmetric plane partitions.
\end{abstract}

\maketitle

\section{Introduction}
\label{sec: intro}

A \defn{plane partition} $\pi$ is an array $(\pi_{i,j})$ of non-negative integers with all but finitely many entries equal to $0$, which is weakly decreasing along rows and columns, i.e., $\pi_{i,j} \geq \pi_{i+1,j}$ and $\pi_{i,j} \geq \pi_{i,j+1}$; see Figure~\ref{fig: PP} (left) for an example. MacMahon \cite{MacMahon97} introduced them at the end of the 19th century as two dimensional generalisations of ordinary partitions and proved in \cite{MacMahon16} two enumeration results:
He showed that the generating function of plane partitions is given by
\begin{equation}
\sum_{\pi} q^{|\pi|} = \prod_{i\geq 1} \frac{1}{(1-q^i)^i},
\end{equation} 
where the sum is over all plane partitions and $|\pi|$ is defined as the sum of the entries of $\pi$.
A plane partition $\pi$ is said to be \defn{contained in an $(a,b,c)$-box} if the entries of $\pi$ are at most $c$ and $\pi_{i,j} \neq 0$ implies $i\leq a$ and $j \leq b$. The plane partition in Figure~\ref{fig: PP} is contained in a $(3,4,4)$-box or any box of larger size. MacMahon showed that the weighted enumeration of plane partitions inside an $(a,b,c)$-box is given by
\begin{equation}
\sum_{\pi} q^{|\pi|} = \prod_{i=1}^a\prod_{j=1}^b\prod_{k=1}^c \frac{1-q^{i+j+k-1}}{1-q^{i+j+k-2}},
\end{equation}
where the sum is over all plane partitions contained in an $(a,b,c)$-box. 
\medskip

\begin{figure}[h]
\begin{center}
\begin{tikzpicture}[scale=0.7]
\node at (0,2) {$4$};
\node at (1,2) {$3$};
\node at (2,2) {$3$};
\node at (3,2) {$1$};
\node at (0,1) {$4$};
\node at (1,1) {$2$};
\node at (2,1) {$1$};
\node at (0,0) {$2$};

\begin{scope}[xshift=8cm, scale=.8]
\PlanePartition{{4,3,3,1},{4,2,1,0},{2,0,0}}
\end{scope}

\begin{scope}[xshift=16cm, scale=.8]
\TilingBox{4}{3}{4}{3}
\PlanePartitionWhite{{4,3,3,1},{4,2,1,0},{2,0,0}}
\end{scope}
\end{tikzpicture}
\caption{\label{fig: PP} A plane partition contained in a $(3,4,4)$-box on the left, its graphical representation as stacks of unit cubes (middle) and the associated lozenge tiling (right).}
\end{center}
\end{figure}
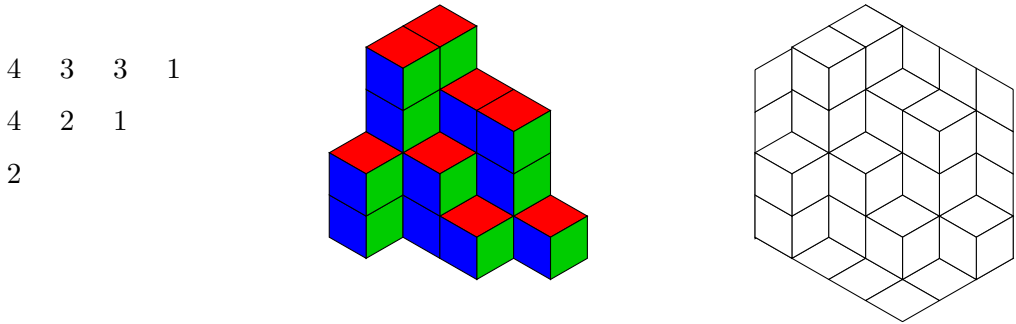

While plane partitions were already introduced at the end of the 19th century, they came into the focus of the combinatorics community mainly in the second half of the last century. One major goal was to prove the enumeration formulas for the ten symmetry classes of plane partitions. These classes are defined via combinations of the three operations \defn{reflection}, \defn{rotation} and \defn{complementation} which are best defined by viewing plane partitions as lozenge tilings: First, we represent a plane partition $\pi$ as stacks of unit cubes by placing $\pi_{i,j}$ unit cubes at the position $(i,j)$, see Figure~\ref{fig: PP} (middle). By further displaying the shape of the $(a,b,c)$-box in which we regard the plane partition in and forgetting the shading of the cubes, we obtain a lozenge tiling of a hexagon with side lengths $a,b,c,a,b,c$, see Figure~\ref{fig: PP} (right). Interestingly, this was first observed by David and Tomei \cite{DavidTomei89} in 1989. The operation \defn{reflection} is defined as vertical reflection of the lozenge tiling, \defn{rotation} as rotation by $120$ degrees and \defn{completion} as rotation by $180$ degree. While MacMahon \cite{MacMahon99} already considered plane partitions invariant under reflection, the operation completion was first described by Mills, Robbins and Rumsey \cite{MillsRobbinsRumsey86} in 1986. A systematic study of the ten symmetry classes which are defined through these operations was initiated by Stanley %\footnote{One should however point out that the enumeration formulas for \TODO{} classes of plane partitions were already proved by ..., and the formulas for \TODO{} further classes were conjectured by \TODO{}.}
 \cite{Stanley86a, Stanley86} and finished in 2011 by Koutschan, Kauers and Zeilberger \cite{KoutschanKauersZeilberger11}. For a more detailed overview see \cite{Krattenthaler16}.
\medskip

Already in 1916, MacMahon \cite{MacMahon16} introduced a further generalisation of partitions to arbitrary dimension, namely \defn{higher-dimensional partitions}. A \defn{$d$-dimensional partition} $\pi$ is an array $(\pi_{i_1,\ldots,i_d})$ of non-negative integers with all but finitely many entries equal to $0$, such that $\pi_{\ii} \geq \pi_{\ii+e_k}$ for all indices $\ii=(i_1,\ldots,i_d)$ and $1 \leq k \leq d$, where $e_k$ denotes the $k$-th unit vector.  We say that $\pi$ is contained in an $(n_1,\ldots,n_{d+1})$-box if all entries are at most $n_{d+1}$ and $\pi_{\ii}>0$ implies that $i_j \leq n_j$ for all $1 \leq j \leq d$. 
Contrary to dimension $1$ (partitions) and dimension $2$ (plane partitions), there are hardly any results known for higher-dimensional partitions in dimension $3$ or higher. MacMahon conjectured a generating formula for each dimension $d$ but it was disproved by Atkin, Bratley, Macdonald and McKay \cite{AtkinBratleyMacdonaldMcKay67} in 1967; see also \cite{Knuth70}. Only recently the first enumeration result for higher-dimensional partitions was presented by Amanov and Yeliussizov \cite{AmanovYeliussizov20Arxiv}. They were able to ``correct'' MacMahons formula and showed that
\begin{equation}
\sum_{\pi} t^{\text{cor}(\pi)} q^{|\pi|_{\text{ch}}} = \prod_{i \geq 1} (1-t q^i)^{-\binom{i+d-2}{d-1}},
\end{equation} 
where the sum is over all $d$-dimensional partitions, and $\text{cor}$ and $|\cdot|_{\text{ch}}$ are certain statistics defined in \cite[Section 4 and 5]{AmanovYeliussizov20Arxiv}. \medskip

In this paper we introduce a new symmetry class for plane partitions, namely \defn{quarter complementary plane partitions} (QCPPs), which can be generalised immediately to higher-dimensional partitions. Instead of presenting the definition for QCPPs (it follows from the corresponding definition for higher-dimensional partitions in Section~\ref{sec: gen fct}) we aim to convey the geometric intuition of this symmetry class next.

\begin{figure}
\begin{center}
\begin{tikzpicture}
\draw (0,0) -- (0,2) -- (2,2) -- (2,0) -- (0,0);
\draw (2,0) -- ({2+1.3*cos(30)},{1.3*sin(30)}) -- ({2+1.3*cos(30)},{2+1.3*sin(30)}) -- (2,2);
\draw (0,2) -- ({1.3*cos(30)},{2+1.3*sin(30)}) -- ({2+1.3*cos(30)},{2+1.3*sin(30)});
\draw[dashed] ({1.3*cos(30)},{1.3*sin(30)}) -- (0,0);
\draw[dashed] ({1.3*cos(30)},{1.3*sin(30)}) -- ({2+1.3*cos(30)},{1.3*sin(30)});
\draw[dashed] ({1.3*cos(30)},{1.3*sin(30)}) -- ({1.3*cos(30)},{2+1.3*sin(30)});
\draw[->, orange, line width=1.5pt] ({1.3*cos(30)},{1.3*sin(30)}) -- ({.75+1.3*cos(30)},{1.3*sin(30)});
\draw[->, cyan, line width=1.75pt] ({1.3*cos(30)},{1.3*sin(30)}) -- ({1.3*cos(30)},{1+1.3*sin(30)});
\draw[->, teal!75!black, line width=1.75pt] ({1.3*cos(30)},{1.3*sin(30)})  -- ({.3*cos(30)},{.3*sin(30)});

\draw[->, orange, line width=1.5pt] (2,2) -- (1.25,2);
\draw[->, cyan, line width=1.75pt] (2,2) -- (2,1);
\draw[->, teal!75!black, line width=1.75pt] (2,2)  -- ({2+cos(30)},{2+sin(30)});

\begin{scope}[xshift=5cm]

\draw (.6,0) -- (1.4,0);
\draw (.6,2) -- (1.4,2);
\draw ({.5+1.5*cos(30)},{2+1.3*sin(30)}) -- ({1.4+1.3*cos(30)},{2+1.3*sin(30)});
\draw[dashed] ({.5+1.5*cos(30)},{1.3*sin(30)}) -- ({1.4+1.3*cos(30)},{1.3*sin(30)});

\draw (0,.3) -- (0,1.8);
\draw (2,.3) -- (2,1.8);
\draw ({2+1.3*cos(30)},{.3+1.3*sin(30)}) -- ({2+1.3*cos(30)},{1.8+1.3*sin(30)});
\draw[dashed] ({1.3*cos(30)},{.3+1.3*sin(30)}) -- ({1.3*cos(30)},{1.8+1.3*sin(30)});

\draw ({2+.9*cos(30)},{.9*sin(30)}) -- ({2+.5*cos(30)},{.5*sin(30)});
\draw ({.9*cos(30)},{2+.9*sin(30)}) -- ({.5*cos(30)},{2+.5*sin(30)});
\draw ({2+.9*cos(30)},{2+.9*sin(30)}) -- ({2+.5*cos(30)},{2+.5*sin(30)});

\draw[dashed] ({.9*cos(30)},{.9*sin(30)}) -- ({.5*cos(30)},{.5*sin(30)});
\node at (0,0) {$(a,1,1)$};
\node at (2,0) {$(a,b,1)$};
\node at (0,2) {$(a,1,c)$};
\node at (2,2) {$(a,b,c)$};
\node at ({1.3*cos(30)},{1.3*sin(30)}) {$(1,1,1)$};
\node at ({1.3*cos(30)},{2+1.3*sin(30)}) {$(1,1,c)$};
\node at ({2+1.3*cos(30)},{1.3*sin(30)}) {$(1,b,1)$};
\node at ({2+1.3*cos(30)},{2+1.3*sin(30)}) {$(1,b,c)$};
\end{scope}

\begin{scope}[xshift=10cm]
\draw (0,0) -- (0,2) -- (2,2) -- (2,0) -- (0,0);
\draw (2,0) -- ({2+1.3*cos(30)},{1.3*sin(30)}) -- ({2+1.3*cos(30)},{2+1.3*sin(30)}) -- (2,2);
\draw (0,2) -- ({1.3*cos(30)},{2+1.3*sin(30)}) -- ({2+1.3*cos(30)},{2+1.3*sin(30)});
\draw[dashed] ({1.3*cos(30)},{1.3*sin(30)}) -- (0,0);
\draw[dashed] ({1.3*cos(30)},{1.3*sin(30)}) -- ({2+1.3*cos(30)},{1.3*sin(30)});
\draw[dashed] ({1.3*cos(30)},{1.3*sin(30)}) -- ({1.3*cos(30)},{2+1.3*sin(30)});

\draw[->, orange, line width=1.5pt] ({1.3*cos(30)},{1.3*sin(30)}) -- ({.75+1.3*cos(30)},{1.3*sin(30)});
\draw[->, cyan, line width=1.75pt] ({1.3*cos(30)},{1.3*sin(30)}) -- ({1.3*cos(30)},{1+1.3*sin(30)});
\draw[->, teal!75!black, line width=1.75pt] ({1.3*cos(30)},{1.3*sin(30)})  -- ({.3*cos(30)},{.3*sin(30)});

\draw[->, orange, line width=1.5pt] ({2+1.3*cos(30)},{2+1.3*sin(30)}) -- ({1.25+1.3*cos(30)},{2+1.3*sin(30)});
\draw[->, cyan, line width=1.75pt] ({2+1.3*cos(30)},{2+1.3*sin(30)}) -- ({2+1.3*cos(30)},{1+1.3*sin(30)});
\draw[->, teal!75!black, line width=1.75pt] ({2+1.3*cos(30)},{2+1.3*sin(30)})  -- ({2+.3*cos(30)},{2+.3*sin(30)});

\draw[->, orange, line width=1.5pt] (0,2) -- (.75,2);
\draw[->, cyan, line width=1.75pt] (0,2)  -- (0,1);
\draw[->, teal!75!black, line width=1.75pt] (0,2)  -- ({cos(30)},{2+sin(30)});

\draw[->, orange, line width=1.5pt] (2,0) -- (1.25,0);
\draw[->, cyan, line width=1.75pt] (2,0)  -- (2,1);
\draw[->, teal!75!black, line width=1.75pt] (2,0)  -- ({2+cos(30)},{sin(30)});
\end{scope}
\end{tikzpicture}
\caption{\label{fig: QCPP intuition} The labels of the corners of an $(a,b,c)$-box (middle), the corners where a copy of $\pi$ is placed for self-complementary plane partitions (left), and the corners where copies of $\pi$ are placed for quarter complementary plane partitions (right). The colour and lengths of the arrows indicate the orientation of the copies.}
\end{center}
\end{figure}
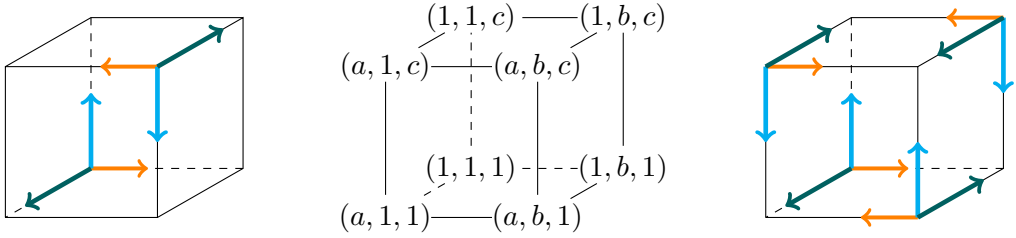

Let $\pi=(\pi_{i,j})$ be a plane partition inside an $(a,b,c)$-box and define by \linebreak $\pi^\prime=(\pi_{a+1-i,b+1-j})_{i,j}$ the ``dual partition'' of $\pi$ inside the $(a,b,c)$-box. Geometrically, we think of the dual partition as stacks of unit cubes hanging from the ceiling of the $(a,b,c)$-box instead of standing at its floor, where the first stack is positioned at the corner with coordinates $(a,b,c)$ instead of the corner with coordinates $(1,1,1)$, see Figure~\ref{fig: QCPP intuition} (left) for a sketch. It is not difficult to see, compare for example with \cite[Section 6]{Krattenthaler16}, that $\pi$ is \defn{self-complementary} if $\pi$ and $\pi^\prime$ fill the $(a,b,c)$-box without overlap when regarded as stacks of unit cubes. In the array perspective this means $\pi_{i,j}+\pi_{i,j}^\prime = c$ for all $1 \leq i \leq a$ and $1 \leq j \leq b$. We can now generalise this idea. Let $C$ be a set of corners of the $(a,b,c)$-box. We define a partition $\pi$ to be \defn{$C$-complementary} if we can fill the $(a,b,c)$-box without overlap by copies of $\pi$ placed at the corners of $C$ similar to before, i.e., if a copy of $\pi$ is placed at a corner of the form $(*,*,c)$ then its stacks of unit cubes hang from the ceiling of the box instead of standing on its floor. It is immediate that such a $\pi$ can only exist if $|C|\in \{1,2,4,8\}$. We can ignore the cases $|C|=1$ and $|C|=8$ since they are trivial. 

If $|C|=2$ and the two corners are on a common face of the box, it is not difficult to see that $\pi$ is determined along the direction orthogonal to this face. Up to rotation there is only one configuration of the two corners such that they do not share a common face: $C=\{(1,1,1),(a,b,c)\}$. The $C$-complementary plane partition in this case are self-complementary plane partitions.
Similarly, if two corners in $C$ are the endpoints of an edge of the box, then $\pi$ is determined along the direction of this edge. The only configuration of four corners for which no pair of them are endpoints of an edge is up to rotation $C=\{(1,1,1),(a,b,1),(1,b,c),(a,1,c)\}$, see Figure~\ref{fig: QCPP intuition} (right) for a sketch. We call the $C$-complementary plane partitions in this case \defn{quarter complementary plane partitions}.
\medskip

In Section~\ref{sec: fcps} we generalise this geometric approach to higher dimensions by using \defn{higher dimensional Ferrers diagrams}. Translating the obtained criteria on Ferrers diagrams back to the array description for higher-dimensional partitions, we obtain in Lemma~\ref{lem: fully comp partitions} the definition of \defn{$d$-dimensional fully complementary partitions} (FCPs). In Proposition~\ref{prop: rec struct} we describe the recursive structure of FCPs which implies immediately our main result.

\begin{thm}
\label{thm: main}
Let $\x=(x_1,\ldots,x_{d+1})$, $\n=(n_1,\ldots,n_{d+1}) \in \N_{>0}^{d+1}$ and denote by $\FCP(\n)$ the set of fully complementary partitions inside a $(2n_1,\ldots,2n_{d+1})$-box. Then
\begin{equation}
\label{eq: main}
\sum_{\mathbf{n} \in \N^{d+1}} |\FCP(\n)| \x^{\n} = \frac{\prod\limits_{i=1}^{d+1}x_i\left(
\sum\limits_{i=1}^{d+1}\left(x_i^{-1}+d x_i\right)-\sum\limits_{1 \leq i, j \leq d+1}x_ix_j^{-1}
\right)
}{\left(1-\sum\limits_{i=1}^{d+1} x_i\right)\prod\limits_{i=1}^{d+1}(1-x_i)}.
\end{equation}
%\begin{equation}
%\label{eq: main}
%\sum_{\mathbf{n} \in \N^{d+1}} \fcp(\n) \x^{\n} = \prod_{i=1}^{d+1}\frac{x_i}{1-x_i} \cdot\frac{\sum\limits_{i=1}^{d+1}\left(x_i^{-1}-1+dx_i\right)+\sum\limits_{1 \leq i < j \leq d+1}x_ix_j^{-1}}{1-\sum\limits_{i=1}^{d+1}x_i}
%\end{equation}
\end{thm}
The recursive structure of FCPs can be further used to construct a bijection between $d$-dimensional FCPs and lattice paths in the positive $d+1$-dimensional orthant starting from any integer point in the interior of its $d$-dimensional boundary.
\bigskip

In Section~\ref{sec: dim 2} we consider the ``typical'' symmetry classes for plane partitions restricted to $2$-dimensional FCPs, i.e. quarter complementary plane partitions. It turns out that there exists at most one symmetric QCPP in an $(a,b,c)$-box and that a QCPP can be neither cyclically symmetric nor transpose-complementary. By introducing two variations, namely \defn{quasi symmetric} and \defn{quasi transpose-complementary}, we are able to show enumerative results for the corresponding symmetry classes and combinations thereof for QCPPs, see Proposition~\ref{prop: QSQCPPs} to Proposition~\ref{prop: TC}. It is an immediate question if the enumeration of plane partitions under these variations of symmetry classes have closed expressions. In Section~\ref{sec: QS pps} we consider these new symmetry classes and similar generalisations. We present three conjectures, the corresponding data generated in computer experiments is given in Appendix~\ref{sec: data}, and the proof of the next result.

\begin{thm}
\label{thm: qtc pps} A plane partition $\pi$ inside an $(n,n,c)$-box is called \defn{quasi transpose-complementary} if $\pi_{i,j}+\pi_{n+1-j,n+1-i}=c$ holds for all $1 \leq i,j \leq n$ with $i \neq n+1-j$.
The number of quasi transpose-complementary plane partitions inside an $(n,n,c)$-box is equal to the number of symmetric plane partitions inside an $(n,n,c)$-box.
\end{thm}

We present different proofs of the above theorem and how it is related to known results on plane partitions, lozenge tilings and perfect matchings respectively. For odd $c$ we relate quasi transpose-complementary plane partitions to transpose-complementary plane partitions and obtain immediately the following numerical connection between symmetric and transpose-complementary plane partitions which seems to be new
\begin{equation}
\label{eq: TCPP and SPP}
2^{n-1} \TCPP(n,n,2c)= \SPP(n-1,n-1,2c+1),
\end{equation}
where $\TCPP(n,n,2c)$ denotes the number of transpose-complementary plane partitions inside an $(n,n,2c)$-box and $\SPP(n,n,c)$ denotes the number of symmetric plane partitions inside an $(n,n,c)$-box.
 Since the proof of the above theorem is computational, it is still an open question if one can find a bijection between symmetric plane partitions and quasi transpose-complementary plane partitions.

%We prove the above result by translating the plane partitions into non-intersecting lattice paths and then using the well-known Lindstr\"om-Gessel-Viennot Theorem. Since we have to make a case distinction with respect to the parity of $c$ we obtain by the Lindstr\"om-Gessel-Viennot Theorem two determinants which already appear among others in the papers \cite{Ciucu97,CiucuFischer15, Proctor88}.  In order to be self-contained we also prove the evaluation of both determinants using an LU decomposition.
%Since one of the determinants is further linked to the enumeration of transpose-complementary plane partitions, we obtained

%Since the proof of the above theorem is computational, it is still an open question if one can find a bijection between the two classes of plane partitions.

\section{Fully complementary partitions}
\label{sec: fcps}
\subsection{Fully complementary Ferrers diagrams}
\label{sec: fc ferrers}

A \defn{$d$-dimensional Ferrers diagram} $\lambda$ is a finite subset of $(\N_{>0})^{d+1}$ such that $(x_1,\ldots,x_{d+1}) \in \lambda$ implies $(y_1,\ldots,y_{d+1}) \in \lambda$ whenever $1 \leq y_i \leq x_i$ for all $1 \leq i \leq d+1$. Equivalently $\lambda$ is an order ideal in the poset $(\N_{>0})^{d+1}$ where the order relation is component-wise the order relation of the integers. For a positive integer $n$, we define $[n]=\{1,\ldots,n\}$.
We say that $\lambda$ is \defn{contained} in an $(n_1,\ldots,n_{d+1})$-box for positive integers $n_1,\ldots, n_{d+1}$, if $\lambda$ is a subset of $[n_1] \times \ldots \times [n_{d+1}]$. It is immediate that the map
\begin{equation*}
\lambda \mapsto \pi(\lambda) = \big(
|\{ (i_1,\ldots,i_d,k)\in \lambda \}|
\big)_{i_1,\ldots,i_d}
\end{equation*}
is a bijection between $d$-dimensional Ferrers diagrams and $d$-dimensional partitions which respects the property of being contained in an $(n_1,\ldots,n_{d+1})$-box. Therefore, we identify for the remainder of this paper a $d$-dimensional Ferrers diagram with the $d$-dimensional partition it is mapped to. \medskip

Let $\n=(n_1,\ldots,n_{d+1})$ be a sequence of positive integers and $I \subseteq [d+1]$. We define the bijection $\rho_{I,\n}$ from $[2n_1] \times \ldots \times [2n_{d+1}]$ onto itself as
\begin{equation*}
\rho_{I,\n}(x_1,\ldots,x_{d+1}) := \left(\begin{cases}
x_i \qquad & i \notin I,\\
2n_i+1-x_i & i \in I,
\end{cases} \right)_{1 \leq i \leq d+1}.
\end{equation*}
If $\n$ is clear from the context we will omit it and write $\rho_I$ instead of $\rho_{I,\n}$. We can use the map $\rho_I$ to rephrase the definition of \emph{quarter complementary plane partitions}. Let $\lambda$ be a $2$-dimensional Ferrers diagram contained in a $(2n_1,2n_2,2n_3)$-box. Then $\rho_{\{1,2\}}(\lambda)$ corresponds to the copy of $\lambda$ placed in the corner $(2n_1,2n_2,1)$ of the box, $\rho_{\{1,3\}}(\lambda)$ to the copy of $\lambda$ placed in the corner $(2n_1,1,2n_3)$ and $\rho_{\{2,3\}}(\lambda)$ to the copy placed in the corner $(1,2n_2,2n_3)$. Hence $\lambda$ corresponds to a QCPP if $\lambda, \rho_{\{1,2\}}(\lambda), \rho_{\{1,3\}}(\lambda), \rho_{\{2,3\}}(\lambda)$ have pairwise empty intersection and their union is equal to $[2n_1]\times [2n_2] \times [2n_3]$. We extend this definition to any dimension $d$.

\begin{defi}
\label{def: FC Ferrers}
Let $n_1,\ldots,n_{d+1}$ be positive integers. A $d$-dimensional Ferrers diagram $\lambda$ is called \defn{fully complementary} inside a $(2n_1,\ldots,2n_{d+1})$-box if for all pairs of subsets $I,J \subseteq [d+1]$ of even size holds $\rho_I(\lambda)  \cap \rho_J(\lambda) = \emptyset$ and
\[
\bigcup_{\substack{I \subseteq [d+1]\\|I| \text{ even}}}\rho_I(\lambda)= [2n_1] \times \ldots \times [2n_{d+1}].
\]
\end{defi}
It is easy to see, that $\rho_I(\lambda)  \cap \rho_J(\lambda) = \emptyset$ holds for all subsets $I,J \subseteq [d+1]$ of even size exactly if $\lambda \cap \rho_I(\lambda) = \emptyset$ holds for all subsets $I \subseteq [d+1]$ of even size.
\bigskip

The $3$-dimensional Ferrers diagram $\lambda=\{(1,1,1,1),(2,1,1,1)\}$ is fully complementary inside a $(2,2,2,2)$-box. This can be seen easily by calculating the according images under $\rho_I$ for even sized $I$ which are given by 
\begin{align*}
&\rho_{\{1,2\}}(\lambda)= \{(2,2,1,1),(1,2,1,1)\}, 
&\rho_{\{2,3\}}(\lambda)= \{(1,2,2,1),(2,2,2,1)\}, \\
&\rho_{\{1,3\}}(\lambda)= \{(2,1,2,1),(1,1,2,1)\}, 
&\rho_{\{2,4\}}(\lambda)= \{(1,2,1,2),(2,2,1,2)\}, \\
&\rho_{\{1,4\}}(\lambda)= \{(2,1,1,2),(1,1,1,2)\}, 
&\rho_{\{3,4\}}(\lambda)= \{(1,1,2,2),(2,1,2,2)\}, \\
&\rho_{\{1,2,3,4\}}(\lambda)= \{(2,2,2,2),(1,2,2,2)\}.
\end{align*}
There are three further Ferrers diagrams which are fully complementary inside $(2,2,2,2)$, namely $\{(1,1,1,1),(1,2,1,1)\}$, $\{(1,1,1,1),(1,1,2,1)\}$ and $\{(1,1,1,1),(1,1,1,2)\}$.
\bigskip

In the above definition of fully complementary we restricted ourselves to boxes with even side lengths. The definition can be extended to any side lengths, however, as we see next, either there are no such Ferrers diagrams, or we obtain a set of Ferrers diagram which is in bijection to one where the box has only even side lengths.
First assume that there exist $k<l$ such that the $k$-th and $l$-th side lengths of the box are given by $2n_k+1$ and $2n_l+1$ respectively and that $\lambda$ is a Ferrers diagram which is fully complementary inside this box. Let us regard the point $P=(1,\ldots,1,n_k,1,\ldots,1,n_l,1,\ldots,1)$ whose components are $1$ except for the $k$-th and $l$-th component. Then there exists an $I\subset [d+1]$ of even size such that $P \in \rho_{I,\widehat{\n}}(\lambda)$, where $\widehat{\n}=(\widehat{n}_1,\ldots,\widehat{n}_{d+1})$ and $\widehat{n}_i$ is half of the side length of the box in the direction of $i$-th standard vector. Let $I^\prime$ be the set $I^\prime = \left(I \setminus \{k,l \} \right) \cup \left( \{k,l\} \setminus I \right)$. It is immediate that $|I^\prime|$ is even and that $P \in \rho_{I^\prime,\widehat{\n}}(\lambda)$ which is a contradiction. Hence there exists no fully complementary Ferrers diagram inside a box with at least two odd side lengths.

Now let all side lengths of the box be even with the exception of one side length. Without loss of generality we can assume that the box is a $(2n_1,\ldots,2n_d,2n_{d+1}+1)$-box. For a fully complementary partition $\pi$ in this box, we see that a point of the form $(i_1,\ldots,i_d,n_{d+1}+1)$ is exactly in $\pi$ if $i_j \leq n_j$ for all $1 \leq j \leq d$. The map 
\[
\Psi: \lambda \mapsto \{ (i_1,\ldots,i_d,i_{d+1}) \in \lambda: i_{d+1} \leq n_{d+1} \} \cup
 \{ (i_1,\ldots,i_d,i_{d+1}-1) \in \lambda: i_{d+1} > n_{d+1}+1 \},
\]
is a surjection from Ferrers diagrams inside a $(2n_1,\ldots,2n_d,2n_{d+1}+1)$-box to Ferrers diagrams inside a $(2n_1,\ldots,2n_d,2n_{d+1})$-box. It is not difficult to see, that the map $\Psi$ commutes with $\rho_I$ for each $I \subseteq [d+1]$, i.e.,
\[
\Psi \circ \rho_{I,(n_1,\ldots,n_d,n_{d+1}+\frac{1}{2})} = \rho_{I,(n_1,\ldots,n_{d+1})} \circ \Psi.
\]
Hence $\Psi$ maps fully complementary Ferrers diagrams inside a $(2n_1,\ldots,2n_d,2n_{d+1}+1)$-box  bijectively to those inside a $(2n_1,\ldots,2n_{d+1})$-box.

\subsection{The generating function of FCPs}
\label{sec: gen fct}

We call a $d$-dimensional partition $\pi$ \defn{fully complementary} inside a $(2n_1,\ldots,2n_{d+1})$-box if its corresponding Ferrers diagram is fully complementary in this box. The fully complementary partitions inside the $(2,2,2,2)$-box are shown next where we write $(\pi_{i,j,1})$ in the top row and $(\pi_{i,j,2})$ in the row below.
\begin{align*}
\begin{matrix}
1 & 1  \\ 
0 & 0
\end{matrix} \qquad \qquad \qquad
\begin{matrix}
1 & 0  \\ 
1 & 0
\end{matrix} \qquad \qquad \qquad
\begin{matrix}
1 & 0  \\ 
0 & 0
\end{matrix} \qquad \qquad \qquad
\begin{matrix}
2 & 0  \\ 
0 & 0
\end{matrix}\\
\\
\begin{matrix}
0 & 0  \\ 
0 & 0
\end{matrix} \qquad \qquad \qquad
\begin{matrix}
0 & 0  \\ 
0 & 0
\end{matrix} \qquad \qquad \qquad
\begin{matrix}
1 & 0  \\ 
0 & 0
\end{matrix} \qquad \qquad \qquad
\begin{matrix}
0 & 0  \\ 
0 & 0
\end{matrix}
\end{align*}
\medskip

For $\n=(n_1,\ldots,n_{d+1})$ and a subset $I \subseteq [d]$ we define the map $\gamma_{I,\n}$ as
\[
\gamma_{I,\n}(\pi) = \left( \pi_{\rho_{I,\n}(i_1,\ldots,i_n)}  \right)_{i_1,\ldots,i_n}.
\]
Again we omit the subscript $\n$ and write $\gamma_I$ whenever $\n$ is clear from context. The next lemma rephrases the conditions of being fully complementary directly for a $d$-dimensional partition.

\begin{lem}
\label{lem: fully comp partitions}
Let $n_1,\ldots, n_{d+1}$ be positive integers. A $d$-dimensional partition $\pi$ is fully complementary inside a $(2n_1,\ldots,2n_{d+1})$-box if and only if
\begin{align}
\label{eq: fc cond 1}
\pi_{i_1,\ldots,i_d} \cdot \gamma_J(\pi)_{i_1,\ldots,i_d}=0, \\
\label{eq: fc cond 2}
\sum_{I \subseteq [d]}\gamma_I(\pi)_{i_1,\ldots,i_d} = 2n_{d+1},
\end{align}
for all non-empty subsets $J \subseteq [d]$ of even size and for all $(i_1,\ldots,i_d) \in [2n_1] \times \ldots \times [2n_d]$.
\end{lem}
\begin{proof}
Let $\lambda$ be a fully complementary Ferrers diagram inside a $(2n_1,\ldots,2n_{d+1})$-box and $(\pi_{i_1,\ldots,i_d})_{i_1,\ldots,i_d}$ the corresponding $d$-dimensional partition. For $I \subseteq [d]$ we have
\[
\gamma_I(\pi)_{i_1,\ldots,i_d}= \begin{cases}
|\{(i_1,\ldots,i_d,k) \in \rho_I (\lambda): k \in \N_{>0}\}| \qquad \qquad & |I| \text{ is even}, \\
|\{(i_1,\ldots,i_d,k) \in \rho_{I\cup\{d+1\}} (\lambda): k \in \N_{>0}\}|  & |I| \text{ is odd}.
\end{cases}
\]
By definition the intersection $\lambda \cap \rho_{I}(\lambda)$ is empty for even sized $I$ exactly if \linebreak$\pi_{i_1,\ldots,i_d} \cdot \gamma_I(\pi)_{i_1,\ldots,i_d}=0$ for all $(i_1,\ldots,i_d) \in [2n_1] \times \ldots \times [2n_d]$. For odd sized $I \subseteq [d]$, the intersection $\lambda \cap \rho_{I \cup \{d+1\}}(\lambda)$ is empty if and only if $\pi_{i_1,\ldots,i_d}+\gamma_I(\pi)_{i_1,\ldots,i_d} \leq 2n_{d+1}$ for all $(i_1,\ldots,i_d) \in [2n_1] \times \ldots \times [2n_d]$.
Finally, the union of all $\rho_J(\lambda)$ with $J \subseteq [d+1]$ of even size is exactly $[2n_1]\times \ldots \times [2n_{d+1}]$ if and only if \eqref{eq: fc cond 2} is satisfied for all $(i_1,\ldots,i_d) \in [2n_1] \times \ldots \times [2n_d]$.
\end{proof}

Denote by $\FCP(n_1,\ldots,n_{d+1})$ the set of fully complementary partitions inside a \linebreak $(2n_1,\ldots,2n_{d+1})$-box.  For $1 \leq k \leq d$ we define the map\footnote{Formally, we have for each $\n=(n_1,\ldots,n_{d+1})$ a different map $\varphi_k$. However since $\n$ will always be clear from the context, we do not include $\n$ in the notation.} $\varphi_k : \FCP(n_1,\ldots,n_{d+1}) \rightarrow \FCP(n_1,\ldots,n_k+1,\ldots,n_{d+1})$ as
\[
\varphi_k(\pi)_{i_1,\ldots,i_d}= 
\begin{cases}
\pi_{i_1,\ldots,i_d} \qquad & i_k \leq n_k, \\
n_{d+1} & i_k \in \{n_k+1,n_k+2\} \text{ and } i_j \leq n_j \text{ for all } 1 \leq j \neq k \leq d,\\
\pi_{i_1,\ldots,i_k-2,\ldots,i_d} & i_k>n_k+2,\\
0 & \text{otherwise},
\end{cases}
\]
and the map $\varphi_{d+1}: \FCP(n_1,\ldots,n_{d+1}) \rightarrow \FCP(n_1,\ldots,n_d,n_{d+1}+1)$ as
\[
\varphi_{d+1}(\pi)_{i_1,\ldots,i_d}= \begin{cases}
\pi_{i_1,\ldots,i_d}+2 \qquad & i_j \leq n_j \text{ for all }1 \leq j \leq d,\\
\pi_{i_1,\ldots,i_d} & \text{ otherwise}.
\end{cases}
\]
It is not difficult to see, that these maps are well-defined.
Let $\pi$ be the fully complementary partition inside the $(4,4,4)$-box displayed next.
\[
\begin{array}{cccc}
4 & 2 & 2 & 0 \\
3 & 2 & 2 & 0 \\
1 & 0 & 0 & 0 \\
0 & 0 & 0 & 0
\end{array}
\]
The images of $\pi$ under the maps $\varphi_1,\varphi_2$ or $\varphi_3$ respectively are given below.
\begin{align*}
\begin{array}{cccc}
4 & 2 & 2 & 0 \\
3 & 2 & 2 & 0 \\
2 & 2 & 0 & 0 \\
2 & 2 & 0 & 0 \\
1 & 0 & 0 & 0 \\
0 & 0 & 0 & 0
\end{array}
\qquad\qquad
\begin{array}{cccccc}
4 & 2 & 2 & 2 & 2  & 0 \\
3 & 2 & 2 & 2 & 2  & 0 \\
1 & 0 & 0 & 0 & 0 & 0 \\
0 & 0 & 0 & 0 & 0 & 0
\end{array}
\qquad\qquad
\begin{array}{cccc}
6 & 4 & 2 & 0 \\
5 & 4 & 2 & 0 \\
1 & 0 & 0 & 0 \\
0 & 0 & 0 & 0
\end{array}
\end{align*}
\bigskip

As we see in a moment, it is useful to extend the definition of fully complementary partitions to ``empty boxes''. In particular we define $\FCP(n_1,\ldots,n_{d+1})$ to consist of the ``empty array'' in case that one $n_{k_0}$ is equal to $0$ and all other $n_i$ are positive. We extend the map $\varphi_k$ to these sets, where $\varphi_k$ is the identity (mapping the empty array onto the empty array) if $k\neq k_0$ and mapping the empty array to
\[
\left(
\begin{cases}
n_{d+1} \qquad &i_j \leq n_j \text{ for all } 1 \leq j \neq k_0 \leq d,\\
0 & \text{ otherwise,}
\end{cases}
\right)_{i_1,\ldots,i_d},
\]
if $k=k_0 \neq d+1$ and to
\[
\left(
\begin{cases}
2 \qquad &i_j \leq n_j \text{ for all } 1 \leq j \leq d,\\
0 & \text{ otherwise,}
\end{cases}
\right)_{i_1,\ldots,i_d},
\]
if $k=k_0=d+1$. The maps $\varphi_k$  allow us to prove the following recursive structure for FCPs.

\begin{prop}
\label{prop: rec struct}
Let $\n=(n_1,\ldots,n_{d+1})$ be a sequence of positive integers. Then $\FCP(\n)$ is equal to the disjoint union
\begin{equation}
\label{eq: FCP recursive}
\FCP(\n) = \dotcup\limits_{1 \leq k \leq d+1} \varphi_k\big(\FCP(\n-e_k) \big),
\end{equation}
where $e_k$ is the $k$-th unit vector.
\end{prop}
\begin{proof}
First we prove that the images of the $\varphi_k$ are disjoint. Let $k<l$ be integers with $\varphi_k\big(\FCP(\n-e_k)\big) \cap \varphi_l\big(\FCP(\n-e_l)\big) \neq \emptyset$ and let $\pi$ be an element of this intersection. First, we assume that $l=d+1$. Since $\pi \in \varphi_k(\n-e_k)$ we have by definition $\pi_{n_1,\ldots,n_d}=n_{d+1}$. On the other hand, $\pi \in \varphi_{d+1}\big(\FCP(\n-e_{d+1})\big)$ implies $\pi_{n_1,\ldots,n_{d}} \geq n_{d+1}+1$ which is a contradiction. Now let $l < d+1$. By definition of $\varphi_k,\varphi_l$ we have 
\[
\pi_{\n+e_k} = \pi_{\n+e_l} =n_{d+1}.
\]
This implies 
\[
\pi_{\n+e_k} \cdot \left( \gamma_{\{k,l\}} (\pi) \right)_{\n+e_k}  =\pi_{\n+e_k} \cdot \pi_{\n+e_l} = n_{d+1}^2 \neq 0,
\]
which contradicts \eqref{eq: fc cond 1}. Hence the union in \eqref{eq: FCP recursive} is disjoint. \medskip

Let $\pi \in \FCP(\n)$. It is easy to see that $\pi_{\n} \geq n_{d+1}$. If $\pi_{\n} \geq n_{d+1}+1$, then $\pi \in \varphi_{d+1}\big(\FCP(\n-e_{d+1})\big)$. Hence let us assume that $\pi_{\n}=n_{d+1}$. By \eqref{eq: fc cond 1}, $\gamma_I(\pi)_{\n}=0$ for all $I \subset [d]$ of even size. By applying any $\gamma_{I}$ for an odd sized $I \subseteq [d]$ to \eqref{eq: fc cond 1}, we see that there is at most one $I$ of odd size for which $\gamma_I(\pi)_{\n}$ is non-zero. Hence by \eqref{eq: fc cond 2} there exists exactly one $I \subseteq [d]$ with $\pi_{\n}+\gamma_I(\pi)_{\n}=2n_{d+1}$ which implies $\gamma_I(\pi)_{\n} = \pi_{\rho_I(\n)}=n_{d+1}$.
Let $k \in I$ and define $\widehat{\n}=(1,\ldots,n_{k},\ldots,1)$ as the vector whose entries are $1$ except on position $k$, where the entry is $n_{k}$ and $\widehat{\ii}=(i_1,\ldots,i_d)$ as the vector with $i_k=n_k$ and $i_j \leq n_j$ for all $1 \leq j \neq k \leq d$. Since $\pi$ is a $d$-dimensional partition, we have the following inequalities.
\begin{center}
\begin{tikzpicture}
\node at (0,2) {$\pi_{\widehat{\n}}$};
\node at (1,2) {$\geq$};
\node at (2,2) {$\pi_{\widehat{\ii}}$};
\node at (3,2) {$\geq$};
\node at (4,2) {$\pi_{\n}$};
\node at (5,2) {$=$};
\node at (6,2) {$n_{d+1}$};
\begin{scope}[yshift=.5cm]
\node[rotate around={90:(-.5,.5)}] (0,0) {$\leq$};
\node[rotate around={90:(.5,1.5)}] (0,0) {$\leq$};
\node[rotate around={90:(1.5,2.5)}] (0,0) {$\leq$};
\node[rotate around={90:(2.5,3.5)}] (0,0) {$=$};
\end{scope}
\node at (0,1) {$\pi_{\widehat{\n}+e_{k}}$};
\node at (1,1) {$\geq$};
\node at (2,1) {$\pi_{\widehat{\ii}+e_{k}}$};
\node at (3,1) {$\geq$};
\node at (4,1) {$\pi_{\n+e_{k}}$};
\node at (5,1) {$\geq$};
\node at (6,1) {$\pi_{\rho_I(\n)}$};
\end{tikzpicture}
\end{center}
Since $\pi_{\rho_{\{k\}}(\n)}=\pi_{\n+e_k}>0$ and there exists only one non-empty $I$ with $\pi_{\rho_I(\n)} \neq 0$, this implies that $I=\{k\}$ and hence $\pi_{\n+e_k}=n_{d+1}$. Furthermore we have $\pi_{\widehat{\n}} +\pi_{\widehat{\n}+e_{k}}\geq \pi_{\n}+\pi_{\n+e_k} =2n_{d+1}$ by the above inequalities and $\pi_{\widehat{\n}} +\pi_{\widehat{\n}+e_{k}}\leq 2n_{d+1}$ by \eqref{eq: fc cond 2} since $\gamma_{\{k\}}(\pi)_{\widehat{\n}}=\pi_{\widehat{\n}+e_{k}}$. This implies $\pi_{\widehat{\n}}= \pi_{\widehat{\n}+e_{k}}=n_{d+1}$ and hence $\pi_{\widehat{i}}=\pi_{\widehat{i}+e_k}=n_{d+1}$ for all $\widehat{i}$ defined as before. Denote for $l \neq k$ by $\widehat{\n}_l$ the vector with all components equal to $1$ except of the $k$-th and $n$-th component which are $n_k$ or $n_l$ respectively. Since $\rho_{\{k,l\}}(\widehat{\n}_l) = \widehat{\n}_l+e_k+e_l$ and $\rho_{\{k,l\}}(\widehat{\n}_l+e_k) = \widehat{\n}_l+e_l$ it follows from  \eqref{eq: fc cond 1} that $\pi_{\widehat{\n}+e_l} =\pi_{\widehat{\n}+e_l+e_k} =0$ and hence $\pi_{i_1,\ldots,i_d}=0$ for all $(i_1,\ldots,i_d)$ such that $i_k\in \{n_k,n_k+1\}$ and there exists an $i_j >n_j$. Denote by $\pi^\prime$ the array obtained by deleting all entries for which the $i$-th component of the index is either $n_k$ or $n_k+1$. It is not difficult to verify that $\pi^\prime \in \FCP(\n-e_{k})$ and $\pi = \varphi_{k}(\pi^\prime)$ which proves the claim.
\end{proof}

As we see next, Theorem~\ref{thm: main} is now a direct consequence of the above proposition.
\begin{proof}[Proof of Theorem~\ref{thm: main}]
Let us denote by $Z(\x)$ the left hand side of  \eqref{eq: main}, i.e., $Z(\x) = \sum_{\n \in \N^{d+1}} |\FCP(\n)|\x^{\n}$. Denote further by $Z_j(\x)= \sum_{\n}|\FCP(\n)|\x^{\n}$, where the sum is over all $\n \in \N^{d+1}$ where the $j$-th component is $0$ and all the other components are positive. It is immediate that
\[
Z_j(\x) = \prod_{\substack{1 \leq i \leq d+1 \\ i\neq j}}\frac{x_i}{1-x_i},
\]
since each of the $\FCP(\n)$ in the sum has exactly one element, namely the ``empty array''. Using Proposition~\ref{prop: rec struct}, we rewrite $Z(\x)$ as
\begin{multline*}
Z(\x) = \sum_{\n \in \N_{>0}^{d+1}}\left|\FCP(\n)\right| \x^{\n} + \sum_{j=1}^{d+1} Z_j(\x)
=  \sum_{\n \in \N_{>0}^{d+1}}\sum_{i=1}^{d+1}\left| \varphi_{i}^{-1}(\FCP(\n))\right| \x^{\n} + \sum_{j=1}^{d+1} Z_j(\x) \\
=  \sum_{i=1}^{d+1} x_i \left( \sum_{\n \in \N_{>0}^{d+1}}|\FCP(\n-e_i)|\x^{\n-e_i} \right) + \sum_{j=1}^{d+1} Z_j(\x)\\
=\sum_{i=1}^{d+1}x_i \left(Z(\x) - \sum_{\substack{1 \leq j \leq n \\ j\neq i}}Z_j(\x) \right)+ \sum_{j=1}^{d+1} Z_j(\x)  \\
= Z(\x) \sum_{i=1}^{d+1}x_i +\sum_{j=1}^{d+1} Z_j(\x)\left(1-\sum_{\substack{1 \leq i \leq d+1 \\ i \neq j}}x_i \right).
\end{multline*}
By bringing all $Z(\x)$ terms on one side and using the explicit formula for $Z_j(\x)$ we obtain the assertion.
\end{proof}

\begin{rem}
\label{rem: FCP to paths}
Let $\n=(n_1,\ldots,n_{d+1}) \in \N_{>0}^{d+1}$ be given. By applying Proposition~\ref{prop: rec struct} iteratively, we see that each $\pi \in \FCP(\n)$ can be uniquely written as $\varphi_{i_k} \circ \ldots \circ \varphi_{i_1} (\sigma)$ where $\sigma$ is an ``empty array'' inside an appropriate ``empty box'' with dimension $\n^\prime=(n_1^\prime,\ldots,n_{d+1}^\prime)$ such that $\varphi_{i_1}(\sigma)$ is a non empty array. We map $\pi$ to the lattice path starting at $\n^\prime$ and ending at $\n$ whose $j$-th step is $e_{i_j}$. This yields a bijection between FCPs inside a $(2n_1,\ldots,2n_{d+1})$-box and lattice paths inside the positive $(d+1)$-dimensional orthant starting from any integer point on its $d$-dimensional boundary and ending at $\n$ with step set $\{e_1,\ldots,e_{d+1}\}$ such that all coordinates are positive after the first step. Below we show the construction for an FCP inside a $(6,4,4)$-box which is mapped to the lattice path from $(1,1,0)$ to $(3,2,2)$ with steps $(e_3,e_1,e_3,e_2,e_1)$.

\[
\begin{array}{cccc}
4 & 2 & 2 & 0 \\
3 & 2 & 2 & 0 \\
2 & 2 & 0 & 0 \\
2 & 2 & 0 & 0 \\
1 & 0 & 0 & 0 \\
0 & 0 & 0 & 0
\end{array}
\hspace{6pt} \xleftarrow{\varphi_1} \hspace{6pt} 
\begin{array}{cccc}
4 & 2 & 2 & 0 \\
3 & 2 & 2 & 0 \\
1 & 0 & 0 & 0 \\
0 & 0 & 0 & 0
\end{array}
\hspace{6pt}  \xleftarrow{\varphi_2} \hspace{6pt} 
\begin{array}{cc}
4 & 0 \\
3 & 0 \\
1 & 0 \\
0 & 0
\end{array}
\hspace{6pt}  \xleftarrow{\varphi_3} \hspace{6pt} 
\begin{array}{cc}
2 & 0 \\
1 & 0 \\
1 & 0 \\
0 & 0
\end{array}
\hspace{6pt}  \xleftarrow{\varphi_1} \hspace{6pt} 
\begin{array}{cc}
2 & 0 \\
0 & 0
\end{array}
\hspace{6pt}  \xleftarrow{\varphi_3} \hspace{6pt} 
\emptyset
\]

\end{rem}

\section{Symmetry classes of QCPPs}
\label{sec: dim 2}

\subsection{(Quasi) symmetric QCPPs}
Remember that a plane partition $\pi$ is quarter complementary inside a $(2a,2b,2c)$-box if for all $1 \leq i \leq 2a$ and $1 \leq j \leq 2b$ we have
\begin{equation}
\label{eq: quart comp 1}
\pi_{i,j} \cdot \pi_{2a+1-i,2b+1-j} = 0,
\end{equation}
and for $\pi_{i,j}>0$  exactly one of the following equations holds: 
\begin{equation}
\label{eq: quart comp 2}
\pi_{i,j}+\pi_{2a+1-i,j}=2c, \qquad \text{or} \qquad \pi_{i,j}+\pi_{i,2b+1-j}=2c.
\end{equation}
Regarded as a regular plane partition, $\pi$ is called symmetric if $a=b$ and  $\pi_{i,j}=\pi_{j,i}$ for all $1 \leq i,j \leq 2a$. By \eqref{eq: quart comp 1} we see that $\pi$ can only be symmetric if the entries on its anti-diagonal are $0$, i.e., $\pi_{i,2a+1-i}=0$ for $1 \leq i \leq 2a$. Together with \eqref{eq: quart comp 2} this implies $\pi_{a,a}=2c$ and hence that the only symmetric quarter complementary plane partition is
\[
\pi= \left(\begin{cases}
2c \qquad &i,j \leq a,\\
0 & \text{otherwise,}
\end{cases}
\right)_{1 \leq i,j \leq 2a}.
\]
In order to obtain more interesting objects we omit the symmetry condition on the anti-diagonal and call a quarter complementary plane partition $\pi$ \defn{quasi symmetric} if $\pi_{i,j}=\pi_{j,i}$ for all $1 \leq i,j \leq 2a$ and $i \neq 2a+1-j$. Denote by $\QS(a,c)$ the set of quarter complementary plane partitions inside a $(2a,2a,2c)$-box which are quasi symmetric.

\begin{prop}
\label{prop: QSQCPPs}
Let $a,c$ be positive integers. Then $\QS(a,c)$ is equal to
\begin{equation}
\label{eq: QSQCPPs}
\QS(a,c) = \varphi_1 \circ \varphi_2 \big(\QS(a-1,c)\big) \dot\cup \varphi_2 \circ \varphi_1 \big(\QS(a-1,c)\big) \dot\cup \varphi_3\big(\QS(a,c-1)\big)
\end{equation}
\end{prop}
\begin{proof}
Let $\pi \in \QS(a,c)$. By Proposition~\ref{prop: rec struct} $\pi$ is either in $\varphi_3\big(\FCP(a,a,c-1)\big)$,\linebreak $\varphi_1\big(\FCP(a-1,a,c)\big)$ or $\varphi_2\big(\FCP(a,a-1,c)\big)$. In the first case $\varphi_3^{-1}(\pi)$ is obviously quasi symmetric again. Assume $\pi \in \varphi_1\big(\FCP(a-1,a,c)\big)$. By the definition of $\varphi_1$ we have $\pi_{a,j}=\pi_{a+1,j}=c$ if $j \leq a$ and $\pi_{a,j}=\pi_{a+1,j}=0$ if $j>a$. By the quasi symmetry of $\pi$ this implies $\pi_{j,a}=\pi_{j,a+1}=c$ for $j < a$ and $\pi_{j,a}=\pi_{j,a+1}=0$ for $j > a+1$. Therefore we obtain $\varphi_1^{-1}(\pi) \in \varphi_2\big(\FCP(a-1,a-1,c)\big)$. It is not difficult to check that $\varphi_1^{-1}(\varphi_2^{-1}(\pi))$ is again quasi symmetric. The last case follows analogously.
\end{proof}
By defining $\QS(a,c)$ to consist of the ``empty array'' if either $a$ or $c$ is equal to $0$, we obtain immediately the following corollary.
\begin{cor}
The generating function for quasi symmetric quarter complementary plane partitions is given by
\[
\sum_{a,c \geq 0} |\QS(a,c)|x^ay^c = \frac{x+y-2x^2-x y}{(1-x)(1-	2x-y)}.
\]
\end{cor}

\subsection{Cyclically symmetric QCPPs}
 A plane partition $\pi$ is called \defn{cyclically symmetric} if a point $(i,j,k)$ in its Ferrers diagram implies that $(j,k,i)$ is also in its Ferrers diagram. As we see next, there is no cyclically symmetric quarter complementary plane partition. Let $\pi$ be a cyclically symmetric quarter complementary Ferrers diagram inside a $(2a,2a,2a)$-box. It is not difficult to see, that being quarter complementary implies that one of the points $(1,1,2a)$, $(1,2a,1)$ or $(2a,1,1)$ has to be part of $\pi$, and hence all of them since $\pi$ is cyclically symmetric. The set $\rho_{\{1,2\}}(\pi)$ therefore contains the points $(2a,1,1)$ and $(1,2a,1)$ which is a contradiction to $\pi \cap \rho_{\{1,2\}}(\pi) = \emptyset$.
 
 Contrary to the previous subsection, we did not find an ``interesting'' generalisation of cyclically symmetric to ``quasi cyclically symmetric'' for which we can deduce either a result or a conjecture in the case of QCPPs.

\subsection{Self- and transpose-complementary QCPPs}

In order to study self- or transpose-complementary QCPPs, we need to specify the box we want to consider the complementation in. For a QCPP inside a $(2a,2b,2c)$-box, the possible boxes for complementation are a $(2a,2b,c)$-, a $(2a,b,2c)$- and an $(a,2b,2c)$-box. For symmetry reasons it suffices to consider QCPPs which are self- or transpose-complementary inside a $(2a,2b,c)$-box.
We call a QCPP $\pi$ inside a $(2a,2b,2c)$-box \defn{self-complementary} if $\pi_{i,j}+\pi_{2a+1-i,2b+1-j}=c$ for all $1 \leq i \leq 2a$ and $1 \leq j \leq 2b$, and \defn{quasi transpose-complementary}\footnote{It is immediate that there are no transpose-complementary QCPPs since the condition $\pi_{i,j}+\pi_{2a+1-j,2a+1-i}=c$ on the anti-diagonal contradicts \eqref{eq: fc cond 1}. Similar to quasi symmetric, we therefore exclude the condition on the anti-diagonal to obtain interesting objects.} if $a=b$ and $\pi_{i,j}+\pi_{2a+1-j,2a+1-i}=c$ for all $1 \leq i,j \leq 2a$ with $i \neq 2a+1-j$. We obtain the following enumeration results.

\begin{prop}
\label{prop: SC}
The number of self-complementary QCPPs inside a $(2a,2b,2c)$-box is $\binom{a+b}{a}$.
\end{prop}
\begin{proof}
The condition $\pi_{i,j}+\pi_{2a+1-i,2b+1-j}=c$ implies that all entries are at most $c$. Together with \eqref{eq: quart comp 2} this implies that all entries are either $0$ or $c$. Hence by Proposition~\ref{prop: rec struct}, $\pi$ is either of the form $\varphi_1(\sigma)$ or $\varphi_2(\sigma)$ for an appropriate QCPP $\sigma$. It is easy to check that $\sigma$ is again self-complementary. The assertion follows now by induction on $a+b$.
\end{proof}

\begin{prop}
\label{prop: TC} A QCPP inside a $(2a,2a,2c)$-box is quasi transpose-complementary if it is quasi symmetric and self-complementary. The number of quasi transpose-complementary QCPPs inside a $(2a,2a,2c)$-box is $2^{a}$.
\end{prop}
\begin{proof}
The condition $\pi_{i,j}+\pi_{2a+1-j,2a+1-i}=c$ together with \eqref{eq: quart comp 2} implies that all entries of $\pi$ are either $0$ or $c$. Since $\pi$ is quarter complementary inside a $(2a,2a,2c)$-box, the sum over all entries must be $2 a^2c$. Hence exactly $2a^2$ entries are equal to $c$ and $2a^2$ entries are equal to $0$. For $i \leq a$ exactly one of the equations $\pi_{i,i}+\pi_{i,2a+1-i}=2c$ or $\pi_{i,i}+\pi_{2a+1-i,i}=2c$ holds by \eqref{eq: fc cond 2}. Hence exactly half of the entries on the anti-diagonal are equal to $c$. For $i \neq 2a+1-j$ we therefore have $\pi_{2a+1-j,2a+1-i}=c$ exactly if $\pi_{i,j}=0$, which is by \eqref{eq: quart comp 1} equivalent to $\pi_{2a+1-i,2a+1-j}=c$. This implies that $\pi$ is quasi symmetric and therefore also self-complementary. It is immediate that a quasi symmetric, self-complementary QCPP is also quasi transpose-complementary.

By the proof of Proposition~\ref{prop: QSQCPPs} and Proposition~\ref{prop: SC} each quasi transpose-complementary QCPP $\pi$ is of the form $\varphi_1 \circ \varphi_2 (\sigma)$ or $\varphi_2 \circ \varphi_1 (\sigma)$, where $\sigma$ is a quasi transpose-complementary QCPP inside a $\big(2(a-1),2(a-1),2c\big)$-box. The assertion follows now by induction on $a$.
\end{proof}

\section{Quasi symmetry classes of plane partitions}
\label{sec: QS pps}

In the previous section we introduced variations of two symmetry classes for quarter complementary plane partitions. The aim of this section is to consider these and similar symmetry classes for plane partitions. The following three conjectures were found by computer experiments. We have added the according data as well as explicit guessed enumeration formulas for small values of one of the parameters in the appendix.

\begin{conj}
\label{conj: quasi sym pps}
Let us denote by $\qspp(a,c)$ the number of quasi symmetric plane partitions inside an $(a,a,c)$-box. Then,
\begin{equation}
\qspp(a,c-a) =\begin{cases}
c \binom{c+a-1}{2a-1}p_a(c) \qquad & a \text{ is even},\\
\binom{c+a-1}{2a-1}p_a(c) \qquad & a \text{ is odd},\\
\end{cases}
\end{equation}
where $p_a(c)$ is an irreducible polynomial in $\mathbb{Q}[c]$ that is even, i.e., $p_a(c)=p_a(-c)$. Further the common denominator of the coefficients of $p_a(c)$ is a product of ``small primes''.
\end{conj}

\begin{conj}
We call a plane partition $\pi$ inside an $(a,a,c)$-box \defn{quasi transpose complementary of second kind} (QTC2), if $\pi$ is transpose complementary except along the diagonal, i.e., $\pi_{i,j}+\pi_{a+1-j,a+1-i}=c$ for all $1 \leq i,j \leq a$ with $i \neq j$. Then for $a\geq 2$, the number $\qtcpp_2(a,c)$  of QTC2 plane partitions inside an $(a,a,c)$-box is given by \begin{equation}
\qtcpp_2(a,c-\frac{a}{2}) = \begin{cases}
c \binom{c+\frac{a}{2}-1}{a-1}p_a(c) \qquad & a \text{ is even},\\
\binom{c+\frac{a}{2}-1}{a-1}p_a(c) \qquad & a \text{ is odd},\\
\end{cases}
\end{equation}
where $p_a(c)$ is an irreducible polynomial in $\mathbb{Q}[c]$ that is even. Further the common denominator of the coefficients of $p_a(c)$ is a product of ``small primes''.
\end{conj}

\begin{conj}
Denote by $\qtcspp_2(a,c)$ the number of symmetric QTC2 plane partitions. Then for $a\geq 2$, 
\begin{equation}
\qtcspp_2(a,c-\frac{a}{2})=
 \begin{cases}
c \binom{c+\frac{a}{2}-1}{a-1}p_a(c) \qquad & a \text{ is even},\\
\binom{c+\frac{a}{2}-1}{a-1}p_a(c) \qquad & a \text{ is odd},\\
\end{cases}
\end{equation}
where $p_a(c)$ is an irreducible polynomial in $\mathbb{Q}[c]$ that is even. Further the common denominator of the coefficients of $p_a(c)$ is a product of ``small primes''.
\end{conj}

\subsection{Proof of Theorem~\ref{thm: qtc pps}}

Let $\pi$ be a quasi transpose complementary plane partition (QTCPP) inside an $(n,n,c)$-box. By definition we have $\pi_{n-j,j} \geq \pi_{n-j+1,j} \geq \pi_{n-j+1,j+1}=c-\pi_{n-j,j}$ for each $1 \leq j \leq n-1$ and equivalently $\pi_{n-j+1,j+1}=c-\pi_{n-j,j} \leq c-\pi_{n-j+1,j} \leq \pi_{n-j,j}$. Hence $\pi$ stays a QTCPP if we replace its diagonal entries $\pi_{n-j+1,j}$ by $\max\left(\pi_{n-j+1,j},c-\pi_{n-j+1,j} \right)$ for all $1 \leq j \leq n$. We denote the resulting QTCPP by $\widehat{\pi}$. Denote by $d(\widehat{\pi})$ the number of anti-diagonal entries which are equal to $\widehat{c}=\left\lfloor \frac{c}{2} \right\rfloor$ and define the weight $\omega(\widehat{\pi})$ as
\[
\omega(\widehat{\pi}) = \begin{cases}
2^n \qquad & c\text{ is odd,}\\
2^{n-d(\widehat{\pi})} & c\text{ is even}.
\end{cases}
\]
Since $|\{\pi_{n-j+1,j},c-\pi_{n-j+1,j}\}|=1$ implies that $c$ is even and $\pi_{n-j+1,j}=\widehat{c}$, it is clear that there are $\omega(\widehat{\pi})$ many QTCPPs mapping to $\widehat{\pi}$ by the above map. Hence the number of QTCPPs is equal to the weighted enumeration of QTCPPs $\widehat{\pi}$ whose anti-diagonal entries are at least $\frac{c}{2}$.
In the following we present two (and a half) proofs for the weighted enumeration of these $\widehat{\pi}$.
\medskip

For odd $c$, each $\widehat{\pi}$ corresponds to a plane partition with entries at most $\widehat{c}$ for which the $i$-th row from top has at most $n+1-i$ positiv entries. The number of these plane partitions can be found in \cite[Corollary 4.1]{Proctor88}, compare also with \cite[Corollary 4.1]{CiucuFischer15}. For even $c$, each $\widehat{\pi}$ corresponds to a lozenge tiling in the ``top half'' of an hexagon with side lengths $n,n,c,n,n,c$, where the bottom of the region is a zig-zag shape directly below the centre line of the hexagon, and each \begin{tikzpicture}[scale=.4]\TopBoundary{0}{0}{0}{3} \end{tikzpicture} lozenge at the bottom of the region is weighted by $\frac{1}{2}$, see Figure~\ref{fig: QTCPP to paths} (right) for an example. The number of these tilings is given in \cite[Corollary 4.3]{CiucuFischer15}. The assertion follows in both cases by using the explicit formulas from \cite{CiucuFischer15, Proctor88}.
\bigskip

Denote by $\qtcpp(n,n,c)$ the number of QTCPPs inside an $(n,n,c)$-box and denote by $M(R)$ the number of perfect matchings of a region $R$. We can rephrase the above observations as
\begin{align*}
\qtcpp(n,n,2c) = 2^n M(P^\prime_{n,n,c}), \qquad
\qtcpp(n,n,2c+1) = 2^n M(P_{n,n,c}),
\end{align*}
where the Regions $P_{n,n,c}$ is defined in \cite[Section 4, Fig. 5]{CiucuFischer15} and $P^\prime_{n,n,c}$ is defined in \cite[Section 4, Fig. 10]{CiucuFischer15}. On the one side we have by Ciucus factorization theorem for graphs with reflective symmetry \cite{Ciucu97} the identity
\[
\PP(a,a,2c) = 2^n  M(P^\prime_{n,n,c}) M(P_{n-1,n-1,c}),
\]
where $\PP(n,n,2c)$ denotes the number of plane partitions inside an $(n,n,2c)$-box (compare for example with \cite{Ciucu97} or \cite[Equation (4.16)]{CiucuFischer15}). This implies the assertion for even $c$ if we know it for odd $c$ and vice versa by using the explicit formulas for the number of (symmetric) plane partitions. Further it is well known, that $2^n  M(P^\prime_{n,n,c})=SPP(n,n,2c)$, see \cite[Equation (5.1)]{CiucuKrattenthaler10} or \cite [Equation (2.6)]{CiucuKrattenthaler17} and that $M(P_{n-1,n-1,c}) = \TCPP(n,n,2c)$, where $\TCPP(n,n,2c)$ is the number of transpose-complementary plane partitions inside an $(n,n,2c)$-box (see for example \cite[Section 6]{Ciucu97}). Combining the above and the assertion (which is already proved above) we obtain immediately the numerical connection between SPPs and TCPPs stated in \eqref{eq: TCPP and SPP}.

\bigskip
%Next we interpret a QTCPP $\widehat{\pi}$ as a family of nonintersecting lattice paths: First we regard $\widehat{\pi}$ as a lozenge tiling and draw $n$ lattice paths in the following way. 

Finally we present another proof using non-intersecting lattice paths. We regard $\widehat{\pi}$ as a lozenge tiling as above and draw $n$ lattice paths in the following way. Each path ends at the top right boundary of the hexagon, uses the steps \begin{tikzpicture}[scale=.4]\LeftBoundary{0}{0}{0}{8} \end{tikzpicture} or \begin{tikzpicture}[scale=.4]\TopBoundary{0}{0}{0}{8} \end{tikzpicture} and the $i$-th path from left has length $n+1-i+\widehat{c}$; see Figure~\ref{fig: QTCPP to paths} (left) for an example. By straightening the paths, we obtain non-intersecting lattice paths starting at $A_i=(2i,-i)$ and ending at $E_i=(n+1+i,\widehat{c}-i)$ with north-steps $(0,1)$ and east-steps $(1,0)$, see Figure~\ref{fig: QTCPP to paths} (right). For odd $c$, each family of paths has the same weight, namely $2^n$. Hence the weighted enumeration is therefore by the Lindstr\"om-Gessel-Viennot Theorem equal to
\begin{equation}
\label{eq: det1}
2^n\det_{1 \leq i,j \leq n} \left( \binom{n+\widehat{c}+1-i}{n+1+j-2i} \right).
\end{equation}
For even $c$ we see that an anti-diagonal entry $\pi_{n+1-j,j}=\widehat{c}$ corresponds to the $j$-th path starting with an east-step. Define the points $B^0_i=(2i+1,-i)$ and $B^1_i=(2i,-i+1)$ which are reached from $A_i$ by an east-step or a north-step respectively. By deleting the first step of each path, we obtain a family of non-intersecting lattice paths starting from either $B^0_i$ or $B^1_i$, where the weight is given by $2$ to the power of times we start at $B^1_i$. Using again the Lindstr\"om-Gessel-Viennot Theorem we obtain for the weighted enumeration
\begin{multline}
\label{eq: det2}
\sum_{(b_1,\ldots,b_n) \in \{0,1\}^n}\det_{1 \leq i,j \leq n} \left( 2^{b_i}\binom{n+\widehat{c}-i}{n+j-2i+b_i} \right)\\
= \det_{1 \leq i,j \leq n} \left( \binom{n+\widehat{c}-i}{n+j-2i} +2\binom{n+\widehat{c}-i}{n+j-2i+1} \right),
\end{multline}
where we used the multilinearity of the determinant in the last step.
Both determinants could be evaluated by guessing the corresponding LU decomposition and using the Pfaff-Saalsch\"utz-summation formula, see for example \cite[(2.3.1.3); Appendix (III.2)]{Slater66}; we omit however the details since there is a simpler and more elegant solution as follows.
First we rewrite the determinant in \eqref{eq: det2} as
\begin{equation}
\label{eq: det3}
\det_{1 \leq i , j \leq n} \left(\binom{n+\widehat{c}-i}{n+j-2i} (n+2\widehat{c}-j+1) \right).
\end{equation}
Then both determinants are special cases of determinant evaluation \cite[Equation (3.13)]{Krattenthaler99}
\[
\det_{1 \leq i,j \leq n}\left(\binom{B L_i+A}{L_i+j} \right)
= \frac{\prod\limits_{1 \leq i < j \leq n}(L_i-L_j)}{\prod\limits_{i=1}^n (L_i+n)!}\prod_{i=1}^n\frac{(B L_i+A)!}{((B-1)L_i+A-1)!}\prod_{i=1}^n(A-Bi+1)_{i-1},
\]
 by setting $L_i=n+1-2i$, $B=\frac{1}{2}$ and $A=\widehat{c}+\frac{n+1}{2}$ in the case of \eqref{eq: det1} and $L_i=n-2i$, $B= \frac{1}{2}$ and $A=\widehat{c} + \frac{n}{2}$ in the case of \eqref{eq: det3}.

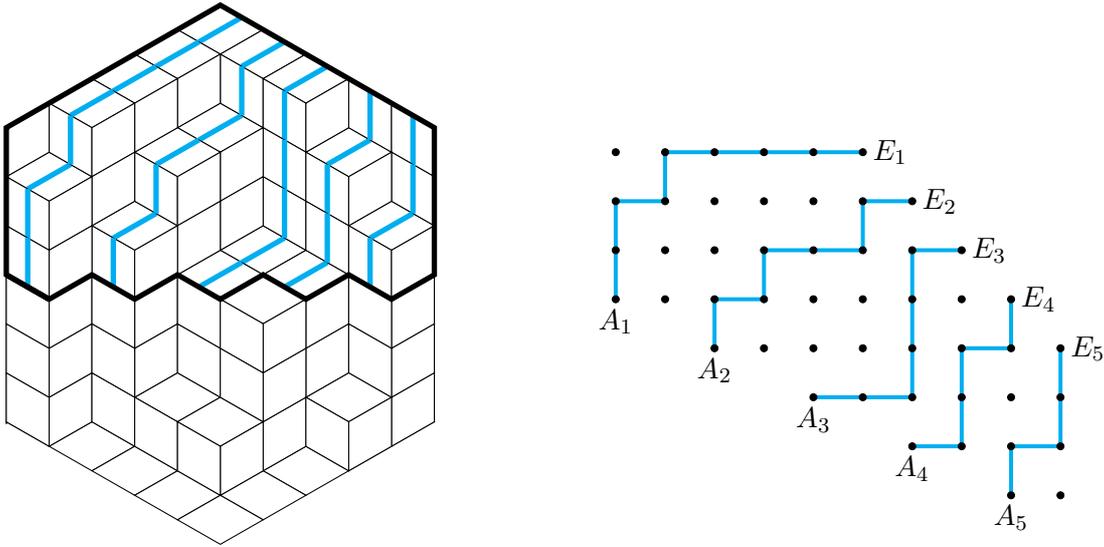
\begin{figure}
\begin{center}
\begin{tikzpicture}[scale=.65]
\TilingBox{5}{5}{6}{3}
\PlanePartitionColour{{6, 6, 6, 5, 4}, {6, 5, 3, 3, 1}, {6, 5, 3, 3, 0}, {6, 4, 1, 1, 0}, {5, 0, 0, 0, 0}}{3}
\TopBoundary{1}{-1}{6}{8}
\TopBoundary{1}{-2}{6}{8}
\TopBoundary{1}{-3}{6}{8}
\TopBoundary{1}{-4}{6}{8}
\TopBoundary{1}{-5}{5}{8}
\TopBoundary{2}{-1}{6}{8}
\TopBoundary{2}{-2}{5}{8}
\TopBoundary{2}{-3}{5}{8}
\TopBoundary{2}{-4}{4}{8}
\TopBoundary{3}{-1}{6}{8}
\TopBoundary{3}{-2}{3}{8}
\TopBoundary{3}{-3}{3}{8}
\TopBoundary{4}{-1}{5}{8}
\TopBoundary{4}{-2}{3}{8}
\TopBoundary{5}{-1}{4}{8}

\LeftBoundary{1}{-4}{6}{8}
\LeftBoundary{1}{-5}{5}{8}
\LeftBoundary{1}{-5}{4}{8}
\LeftBoundary{2}{-1}{6}{8}
\LeftBoundary{2}{-3}{5}{8}
\LeftBoundary{2}{-4}{4}{8}
\LeftBoundary{3}{-1}{6}{8}
\LeftBoundary{3}{-1}{5}{8}
\LeftBoundary{3}{-1}{4}{8}
\LeftBoundary{4}{0}{6}{8}
\LeftBoundary{4}{-1}{5}{8}
\LeftBoundary{4}{-1}{4}{8}
\LeftBoundary{5}{0}{6}{8}
\LeftBoundary{5}{0}{5}{8}
\LeftBoundary{5}{-1}{4}{8}

\RightBoundary{3}{-1}{6}{8}
\RightBoundary{4}{-1}{5}{8}
\RightBoundary{5}{-1}{4}{8}
\RightBoundary{1}{-2}{6}{8}
\RightBoundary{2}{-2}{5}{8}
\RightBoundary{2}{-2}{4}{8}
\RightBoundary{1}{-3}{6}{8}
\RightBoundary{2}{-3}{5}{8}
\RightBoundary{2}{-3}{4}{8}
\RightBoundary{1}{-4}{6}{8}
\RightBoundary{1}{-4}{5}{8}
\RightBoundary{2}{-4}{4}{8}
\RightBoundary{0}{-5}{6}{8}
\RightBoundary{1}{-5}{5}{8}
\RightBoundary{1}{-5}{4}{8}

\draw[line width=2pt] ({-5*cos(30)},{sin(30)}) -- ({-4*cos(30)},0) -- ({-3*cos(30)},{sin(30)}) -- ({-2*cos(30)},0) -- ({-1*cos(30)},{sin(30)}) -- (0,0) -- ({1*cos(30)},{sin(30)}) -- ({2*cos(30)},0) -- ({3*cos(30)},{sin(30)}) -- ({4*cos(30)},0) -- ({5*cos(30)},{sin(30)}) -- ({5*cos(30)},{3+sin(30)}) -- (0,{3+6*sin(30)}) -- ({-5*cos(30)},{3+sin(30)}) -- ({-5*cos(30)},{sin(30)}) -- ({-4*cos(30)},0);

\begin{scope}[xshift=6cm, yshift=1cm]

\node at (2,-1)[below] {$A_1$};
\node at (4,-2)[below] {$A_2$};
\node at (6,-3)[below] {$A_3$};
\node at (8,-4)[below] {$A_4$};
\node at (10,-5)[below] {$A_5$};
\node at (7,2)[right] {$E_1$};
\node at (8,1)[right] {$E_2$};
\node at (9,0)[right] {$E_3$};
\node at (10,-1)[right] {$E_4$};
\node at (11,-2)[right] {$E_5$};

\draw[cyan, line width=1.5pt] (2,-1) -- (2,1) -- (3,1) -- (3,2) -- (7,2);
\draw[cyan, line width=1.5pt] (4,-2) -- (4,-1) -- (5,-1) -- (5,0) -- (7,0) -- (7,1) -- (8,1);
\draw[cyan, line width=1.5pt] (6,-3) -- (8,-3) -- (8,0) -- (9,0);
\draw[cyan, line width=1.5pt] (8,-4) -- (9,-4) -- (9,-2) -- (10,-2) -- (10,-1);
\draw[cyan, line width=1.5pt]  (10,-5) -- (10,-4) -- (11,-4) -- (11,-2);
\foreach \i in {1,...,4}
{
	\foreach \j in {\i,...,8}
		\filldraw ({\j+3},{-\j+\i+2}) circle (2pt);
}
\foreach \i in {1,...,8}
	\filldraw ({\i+2},{3-\i}) circle (2pt);
\foreach \i in {1,...,7}
	\filldraw ({\i+1},{3-\i}) circle (2pt);
\foreach \i in {1,...,5}
	\filldraw ({\i+1},{2-\i}) circle (2pt);
\foreach \i in {1,...,3}
	\filldraw ({\i+1},{1-\i}) circle (2pt);
\filldraw (2,-1) circle (2pt);
\end{scope}
\end{tikzpicture}
\end{center}
\caption{\label{fig: QTCPP to paths} A QTCPP $\widehat{\pi}$ inside a $(5,5,6)$-box whose anti-diagonal entries are at least $3$ (left) and its corresponding lattice path configuration (right).}
\end{figure}

\section{Acknowledgement}

The author thanks Ilse Fischer and Christian Krattenthaler for helpful discussions.

\begin{appendix}
\section{Data from the computer experiments}
\label{sec: data}

\subsection{Quasi symmetric plane partitions}

The values for the number $\qspp(a,c)$ of quasi symmetric plane partitions inside an $(a,a,c)$-box with $a \leq 6$ and $c \leq 10$ are shown in the next table.
\begin{center}
\begin{tabular}{c|rr rr rr}
c \textbackslash\, a & 1 & 2 & 3 & 4 & 5 & 6\\
\hline
0 &  1  & 1    & 1     & 1        & 1          & 1\\
1 &  2  & 6    & 12    & 32       & 60         & 164\\
2 &  3  & 20   & 69    & 400      & 1312       & 7952\\
3 &  4  & 50   & 272   & 3052     & 16572      & 200956\\
4 &  5  & 105  & 846   & 16932    & 145428     & 3284589\\
5 &  6  & 196  & 2232  & 74868    & 979068     & 38963092\\
6 &  7  & 336  & 5214  & 278928   & 5376673    & 360346984\\
7 &  8  & 540  & 11088 & 908336   & 25100880   & 2727638524\\
8 &  9  & 825  & 21879 & 2653001  & 102593290  & 17499041992\\
9 &  10 & 1210 & 40612 & 7081776  & 375222392  & 97667820784 \\
10 & 11 & 1716 & 71643 & 17524416 & 1248707892 & 483901238656
\end{tabular}
\end{center}
Using the above values, we conjecture the following formulas for $\qspp(a,c)$ for $1 \leq a \leq 6$.
\begin{align*}
\qspp(1,c-1) &= \binom{c}{1} ,\\
\qspp(2,c-2) &= c \binom{c+1}{3} ,\\
\qspp(3,c-3) &= \binom{c+2}{5} \frac{c^2-2}{7},\\
\qspp(4,c-4) &= c \binom{c+3}{7} \frac{41 c^4-229 c^2-892}{23760},\\
\qspp(5,c-5) &= \binom{c+4}{9} \frac{683c^8 -8206c^6-14473 c^4-310644 c^2+756000}{122522400},\\
\qspp(6,c-6) &= c \binom{c+5}{11} \frac{1}{161911881331200}\left( 56381 c^{12}- 1850347 c^{10} + 
  11282865 c^8\right. \\
  & \left. - 28759181 c^6- 1859025278 c^4+ 20697349128 c^2 +194655992832\right)
.
\end{align*}

\subsection{Quasi transpose complementary plane partitions of second kind}

For the number $\qtcpp_2(a,c)$ of QTC2 plane partitions inside an $(a,a,c)$-box with $a \leq 6$ and $c\leq 10$, we have the following values.
\begin{center}
\begin{tabular}{c|rr rr rr}
c \textbackslash\, a & 1 & 2 & 3 & 4 & 5 & 6\\
\hline
0 & 1   & 1   & 1    & 1       & 1         & 1\\
1 & 2   & 4   & 7    & 24      & 62        & 216       \\
2 & 3   & 9   & 26   & 224     & 1210      & 12177     \\
3 & 4   & 16  & 70   & 1280    & 12819     & 314624    \\
4 & 5   & 25  & 155  & 5361    & 91694     & 4860048   \\
5 & 6   & 36  & 301  & 18088   & 496796    & 51955744  \\
6 & 7   & 49  & 532  & 52032   & 2185860   & 420545536 \\
7 & 8   & 64  & 876  & 132408  & 8177634   & 2735918368\\
8 & 9   & 81  & 1365 & 305745  & 26861407  & 14918043569\\
9 & 10  & 100 & 2035 & 652432  & 79299714  & --\\
10 & 11 & 121 & 2926 & 1304160 & 214133686 & --\\
\end{tabular}
\end{center}
We conjecture the following explicit formulas for $\qtcpp_2(a,c)$ for $1 \leq a \leq 6$.
\begin{align*}
\qtcpp_2(1,c-\frac{1}{2}) &= \left(c+\frac{1}{2}\right),\\
\qtcpp_2(2,c-1) &= c\binom{c}{1},\\
\qtcpp_2(3,c-\frac{1}{2}) &= \binom{c+\frac{1}{2}}{2} \frac{4c^2+3}{12},\\
\qtcpp_2(4,c-2) &= c\binom{c+1}{3}\frac{5c^4+19c^2-16}{280},\\
\qtcpp_2(5,c-\frac{1}{2}) &= \binom{c+\frac{3}{2}}{4}\frac{54016 c^8+ 522240 c^6- 1322016 c^4 + 80480 c^2 +883575}{159667200},\\
\qtcpp_2(6,c-3) &= c\binom{c+2}{5}\frac{1}{256505356800}\left(73325 c^{12} +1648357 c^{10} -12312285 c^8 \right.\\
&\left.+ 29029591 c^6 + 201378740 c^4- 876526848 c^2 + 395435520\right).
\end{align*}

\subsection{Quasi transpose complementary symmetric plane partitions of second kind}
Computer experiments yield the following values for the number $\qtcspp_2(a,c)$ of symmetric QTC2 plane partitions inside an $(a,a,c)$-box with $a \leq6$ and $c \leq 10$.
\begin{center}
\begin{tabular}{c|rr rr rr}
c \textbackslash\, a & 1 & 2 & 3 & 4 & 5 & 6\\
\hline
0 & 1   & 1   & 1   & 1     & 1      & 1\\
1 & 2   & 4   & 5   & 12    & 18     & 40     \\
2 & 3   & 9   & 14  & 68    & 142    & 625    \\
3 & 4   & 16  & 30  & 260   & 723    & 5728   \\
4 & 5   & 25  & 55  & 777   & 2782   & 36876  \\
5 & 6   & 36  & 91  & 1960  & 8796   & 184224 \\
6 & 7   & 49  & 140 & 4368  & 24036  & 759708 \\
7 & 8   & 64  & 204 & 8856  & 58674  & 2695200\\
8 & 9   & 81  & 285 & 16665 & 130911 & 8468889\\
9 & 10  & 100 & 385 & 29524 & 271414 & --\\
10 & 11 & 121 & 506 & 49764 & 529386 & --\\
\end{tabular}
\end{center}
For $a \leq 6$ we conjecture the following formulas for $\qtcspp_2(a,c)$.
\begin{align*}
\qtcspp_2(1,c)&= c+1,\\
\qtcspp_2(2,c-1)&= c\binom{c}{1},\\
\qtcspp_2(3,c-\frac{3}{2})&= \binom{c+\frac{1}{2}}{2}\frac{2c}{3},\\
\qtcspp_2(4,c-2)&= c\binom{c+1}{3}\frac{c^2+1}{10},\\
\qtcspp_2(5,c-\frac{5}{2})&= \binom{c+\frac{3}{2}}{4}\frac{144 c^4+248c^2-455}{6720},\\
\qtcspp_2(6,c-3)&= c\binom{c+2}{5}\frac{3c^6+15c^4-58c^2+200}{9240},\\
\end{align*}

\end{appendix}

\bibliographystyle{abbrvurl}
\bibliography{LiteraturListe}

\begin{thebibliography}{10}

\bibitem{AmanovYeliussizov20Arxiv}
A.~Amanov and D.~Yeliussizov.
\newblock {MacMahon's statistics on higher-dimensional partitions}.
\newblock {\em \href{https://arxiv.org/abs/2009.00592}{arXiv:2009.00592}},
  2020.

\bibitem{AtkinBratleyMacdonaldMcKay67}
A.~O.~L. Atkin, P.~Bratley, I.~G. Macdonald, and J.~K.~S. McKay.
\newblock Some computations for {$m$}-dimensional partitions.
\newblock {\em Proc. Cambridge Philos. Soc.}, 63:1097--1100, 1967.
\newblock \href {https://doi.org/10.1017/s0305004100042171}
  {\path{doi:10.1017/s0305004100042171}}.

\bibitem{Ciucu97}
M.~Ciucu.
\newblock Enumeration of perfect matchings in graphs with reflective symmetry.
\newblock {\em J. Combin. Theory Ser. A}, 77(1):67--97, 1997.
\newblock \href {https://doi.org/10.1006/jcta.1996.2725}
  {\path{doi:10.1006/jcta.1996.2725}}.

\bibitem{CiucuFischer15}
M.~Ciucu and I.~Fischer.
\newblock Proof of two conjectures of {C}iucu and {K}rattenthaler on the
  enumeration of lozenge tilings of hexagons with cut off corners.
\newblock {\em J. Combin. Theory Ser. A}, 133:228--250, 2015.
\newblock \href {https://doi.org/10.1016/j.jcta.2015.02.008}
  {\path{doi:10.1016/j.jcta.2015.02.008}}.

\bibitem{CiucuKrattenthaler10}
M.~Ciucu and C.~Krattenthaler.
\newblock A factorization theorem for classical group characters, with
  applications to plane partitions and rhombus tilings.
\newblock In I.~S. Kotsireas and E.~V. Zima, editors, {\em Advances in
  combinatorial mathematics: Proceedings of the Waterloo Workshop in Computer
  Algebra 2008}, pages 39--60. Springer, Berlin, 2010.
\newblock \href {https://doi.org/10.1007/978-3-642-03562-3}
  {\path{doi:10.1007/978-3-642-03562-3}}.

\bibitem{CiucuKrattenthaler17}
M.~Ciucu and C.~Krattenthaler.
\newblock A factorization theorem for lozenge tilings of a hexagon with
  triangular holes.
\newblock {\em Trans. Amer. Math. Soc.}, 369(5):3655--3672, 2017.
\newblock \href {https://doi.org/10.1090/tran/7047}
  {\path{doi:10.1090/tran/7047}}.

\bibitem{DavidTomei89}
G.~David and C.~Tomei.
\newblock The problem of the calissons.
\newblock {\em Amer. Math. Monthly}, 96(5):429--431, 1989.
\newblock \href {https://doi.org/10.2307/2325150} {\path{doi:10.2307/2325150}}.

\bibitem{Knuth70}
D.~E. Knuth.
\newblock A note on solid partitions.
\newblock {\em Math. Comp.}, 24:955--961, 1970.
\newblock \href {https://doi.org/10.2307/2004628} {\path{doi:10.2307/2004628}}.

\bibitem{KoutschanKauersZeilberger11}
C.~Koutschan, M.~Kauers, and D.~Zeilberger.
\newblock Proof of {G}eorge {A}ndrews's and {D}avid {R}obbins's {$q$}-{TSPP}
  conjecture.
\newblock {\em Proc. Natl. Acad. Sci. USA}, 108(6):2196--2199, 2011.
\newblock \href {https://doi.org/10.1073/pnas.1019186108}
  {\path{doi:10.1073/pnas.1019186108}}.

\bibitem{Krattenthaler99}
C.~Krattenthaler.
\newblock Advanced determinant calculus.
\newblock {\em S\'{e}m. Lothar. Combin.}, 42:Art. B 42q, 67 pp., 1999.

\bibitem{Krattenthaler16}
C.~Krattenthaler.
\newblock {\em {The mathematical legacy of Richard P. Stanley}}, chapter {Plane
  partitions in the work of Richard Stanley and his school}, pages 231--261.
\newblock Amer. Math. Soc., Providence, RI, 2016.
\newblock \href {https://doi.org/10.1090/mbk/100} {\path{doi:10.1090/mbk/100}}.

\bibitem{MacMahon97}
P.~A. MacMahon.
\newblock {Memoir on the theory of the partition of numbers, I}.
\newblock {\em Lond. Phil. Trans. (A)}, 187:619--673, 1897.

\bibitem{MacMahon99}
P.~A. MacMahon.
\newblock {Partitions of numbers whose graphs possess symmetry}.
\newblock {\em Trans. Cambridge Philos. Soc.}, 17:149--170, 1899.

\bibitem{MacMahon16}
P.~A. MacMahon.
\newblock {\em Combinatory Analysis, vol. 2}.
\newblock Cambridge University Press, 1916; reprinted by Chelsea, New York,
  1960.

\bibitem{MillsRobbinsRumsey86}
W.~H. Mills, D.~P. Robbins, and H.~Rumsey~Jr.
\newblock Self-complementary totally symmetric plane partitions.
\newblock {\em J. Combin. Theory Ser. A}, 42(2):277--292, 1986.
\newblock \href {https://doi.org/10.1016/0097-3165(86)90098-1}
  {\path{doi:10.1016/0097-3165(86)90098-1}}.

\bibitem{Proctor88}
R.~A. Proctor.
\newblock Odd symplectic groups.
\newblock {\em Invent. Math.}, 92(2):307--332, 1988.
\newblock \href {https://doi.org/10.1007/BF01404455}
  {\path{doi:10.1007/BF01404455}}.

\bibitem{Slater66}
L.~J. Slater.
\newblock {\em Generalized hypergeometric functions}.
\newblock Cambridge University Press, Cambridge, 1966.

\bibitem{Stanley86a}
R.~P. Stanley.
\newblock A baker's dozen of conjectures concerning plane partitions.
\newblock In {\em Combinatoire \'{e}num\'{e}rative}, volume 1234 of {\em
  Lecture Notes in Math.}, pages 285--293. Springer, Berlin, 1986.
\newblock \href {https://doi.org/10.1007/BFb0072521}
  {\path{doi:10.1007/BFb0072521}}.

\bibitem{Stanley86}
R.~P. Stanley.
\newblock Symmetries of plane partitions.
\newblock {\em J. Combin. Theory Ser. A}, 43(1):103--113, 1986.
\newblock Erratum 44:310, 1987.
\newblock \href {https://doi.org/10.1016/0097-3165(86)90028-2}
  {\path{doi:10.1016/0097-3165(86)90028-2}}.

\end{thebibliography}

\end{document}